\documentclass[12pt,reqno]{article}

\usepackage[usenames]{color}
\usepackage{amssymb}
\usepackage{graphicx}
\usepackage{amscd}
\usepackage[english]{babel}

\usepackage{float}

\usepackage[colorlinks=true,
linkcolor=webgreen,
filecolor=webbrown,
citecolor=webgreen]{hyperref}

\definecolor{webgreen}{rgb}{0,.5,0}
\definecolor{webbrown}{rgb}{.6,0,0}

\usepackage{color}
\usepackage{fullpage}

\usepackage{psfig}
\usepackage{graphics,amsmath,amssymb}
\usepackage{amsthm}
\usepackage{amsfonts}
\usepackage{latexsym}
\usepackage{epsf}

\newcommand{\beql}[1]{\begin{equation}\label{#1}}
\newcommand{\eeq}{\end{equation}}
\newcommand{\eqn}[1]{(\ref{#1})}

\newcommand{\sA}{{\cal{A}}}
\newcommand{\sD}{{\cal{D}}}
\newcommand{\sM}{{\cal{M}}}

\newtheorem{thm}{Theorem}{\bfseries}{\itshape}
\newtheorem{cor}[thm]{Corollary}{\bfseries}{\itshape}
\newtheorem{lem}[thm]{Lemma}{\bfseries}{\itshape}
{\bfseries}{\itshape}
\newtheorem{conj}{Conjecture}{\bfseries}{\itshape}

\DeclareMathAlphabet{\curly}{U}{rsfs}{m}{n}

\DeclareMathOperator{\len}{len}

\newsavebox{\gplussave}
\sbox{\gplussave}{\raisebox{0ex}{\includegraphics[width=8pt]{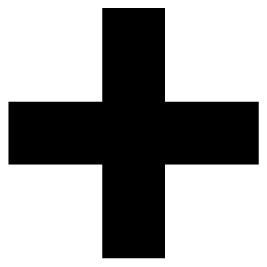}}}
\newcommand{\gplus}{\usebox{\gplussave}}

\newsavebox{\gminussave}
\sbox{\gminussave}{\raisebox{0ex}{\includegraphics[width=8pt]{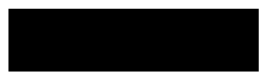}}}

\newsavebox{\gtimessave}
\sbox{\gtimessave}{\raisebox{-.1ex}{\includegraphics[width=9pt]{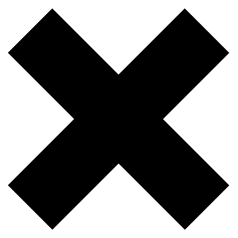}}}
\newcommand{\gtimes}{\usebox{\gtimessave}}

\newcommand{\CA}{\,\gplus\,}
\newcommand{\CAsub}{\,\gplus}

\newcommand{\CM}{\,\gtimes\,}
\newcommand{\CMsub}{\,\gtimes}
\newcommand{\RR}{\mathbb R}
\newcommand{\ZZ}{\mathbb Z}

\usepackage[bottom]{footmisc}

\begin{document}
\theoremstyle{plain}

\begin{center}
{\Large\bf Dismal Arithmetic} \\
\vspace*{+.2in}
David~Applegate, \\
AT\&T Shannon Labs, \\
180 Park Ave., Florham Park, NJ 07932-0971, USA \\
\href{mailto:david@research.att.com}{\tt david@research.att.com} \\  [+.1in]
Marc~LeBrun, \\
Fixpoint Inc., \\
448 Ignacio Blvd. \#239, Novato, CA 94949, USA \\
\href{mailto:mlb@well.com}{\tt mlb@well.com} \\  [+.1in]
N.~J.~A.~Sloane,${}^{(a)}$ \\
AT\&T Shannon Labs, \\
180 Park Ave., Florham Park, NJ 07932-0971, USA \\
\href{mailto:njas@research.att.com}{\tt njas@research.att.com} \\  
\vspace*{+.1in}
${}^{(a)}$ To whom correspondence should be addressed.

\vspace*{+.2in}
\textsc{To the memory of Martin Gardner (October 21, 1914 -- May 22, 2010).}  \\
\vspace*{+.2in}

July 5, 2011

\vspace*{+.1in}
\end{center}

\begin{center}
{\bf Abstract}
\end{center}
Dismal arithmetic is just like the arithmetic you learned in school,
only simpler: there there are no carries,
when you add digits you just take the largest,
and when you multiply digits you take the smallest.  
This paper studies basic number theory in this world,
including analogues of the primes,
number of divisors, sum of divisors, and the partition function.

\section{Introduction}\label{Sec1}

To remedy the dismal state of arithmetic skills possessed by
today's children, we propose a ``dismal arithmetic'' that will be 
easier to learn than the usual version.
It is easier because there are no carry digits
and there is no need to add or multiply digits,
or to do anything harder than comparing.  
In dismal arithmetic, for each pair of digits, \\
\indent
\hspace{.25in}to Add, take the lArger, but\\
\indent
\hspace{.25in}to Multiply, take the sMaller. \\
That's it! For example:
$2 \CA 5=5$,
$2 \CM 5=2$.

Addition or multiplication of larger numbers uses the same rules, 
always with the proviso that there are no carries.
For example, the dismal sum of $169$ and $248$ is $269$ and their dismal product is $12468$
(Figure~1).

\begin{minipage}[t]{2.5in}
\begin{center}
$$
\begin{tabular}{ p{.001in}  p{.001in}  p{.001in}  p{.001in} }
       & 1 & 6 & 9 \\
$\CA$  & 2 & 4 & 8 \\
\hline
       & 2 & 6 & 9 \\
\hline
\end{tabular}
$$
Fig. 1(a) Dismal addition.
\end{center}
  \end{minipage}
  \begin{minipage}[t]{2.5in}
\begin{center}
$$
\begin{tabular}{ p{.001in} p{.001in} p{.001in} p{.001in} p{.001in} }
       &   & 1 & 6 & 9 \\
$\CM$  &   & 2 & 4 & 8 \\
\hline
       &   & 1 & 6 & 8 \\
       & 1 & 4 & 4 &   \\
     1 & 2 & 2 &   &   \\
\hline
     1 & 2 & 4 & 6 & 8 \\
\hline
\end{tabular}
$$
Fig. 1(b) Dismal multiplication.
\end{center}
  \end{minipage}
\vspace*{+.2in}

One might expect that nothing interesting
could arise from such simple rules. However, developing
the dismal analogue of ordinary elementary number
theory will lead us to some surprisingly difficult questions.

Here are a few dismal analogues of standard sequences.
The ``even'' numbers, $2 \CM n$, are
\beql{EqEven}
0, \, 1, \, 2, \, 2, \, 2, \, 2, \, 2, \, 2, \, 2, \, 2, \, 10, \, 11, \, 12, \, 12, \, 12, \, 12, \, 12, \, 12, \, 12, \, 12, \, 20, \, 21, \, 22, \, 22, \, \cdots
\eeq
(entry A171818 in \cite{OEIS}).
Note that $n \CA n$ (which is simply $n$) is a different sequence.
For another, less obvious, analogue of the even numbers, see
\eqn{EqA162672} in \S\ref{SecPrimes}.
The squares, $n \CM n$, are
\beql{EqSquares}
0, \, 1, \, 2, \, 3, \, 4, \, 5, \, 6, \, 7, \, 8, \, 9, \, 100, \, 111, \, 112, \, 113, \, 114, \, 115, \, 116, \, 117, \, 118, \, 119, \, 200, \, \cdots
\eeq
(A087019),\footnote{See also the sums of two squares, A171120.
The numbers $10, 11, \ldots, 99$ are not the sum
of any number of squares, so there is
no dismal analogue of the four-squares theorem.}
the dismal triangular numbers, $0 \CA 1 \CA 2 \ldots \CA n$, are
\beql{EqTriang}
0, \,  1, \,  2, \,  3, \,  4, \,  5, \,  6, \,  7, \,  8, \,  9, \,  19, \,  19, \,  19, \,  19, \,  19, \,  19, \,  19, \,  19, \,  19, \,  19, \,  29, \,  29, \,  29, \,  29, \,  29, \, \ldots
\eeq
(A087052),
and the dismal factorials, $1 \CM 2 \CM \cdots \CM n$, $n \ge 1$, are
\begin{align}\label{EqFact}
~~~ & 1, \,  1, \,  1, \,  1, \,  1, \,  1, \,  1, \,  1, \,  1, \,  10, \,  110, \,  1110, \,  11110, \,  111110, \,  1111110, \,  11111110, \,  111111110, \, \nonumber \\
~~~ & 1111111110, \,  11111111110, \,  111111111100, \,  1111111111100, \,  11111111111100, \, \ldots  
\end{align}
(A189788).

A formal definition of dismal arithmetic is given in
\S\ref{Sec2}, valid for any base $b$, not just base $10$,
and it shown there that the commutative, associative, and distributive
laws hold (Theorem \ref{ThLaws}).
In that section we also introduce the notion
of a ``digit map,'' in order to study how changing individual
digits in a dismal calculation affects the answer (Theorem \ref{Th712},
Corollary \ref{Cor705}).

The dismal primes are the subject of \S\ref{SecPrimes}.
A necessary condition for a number to be
a prime is that it contain a digit equal to $b-1$.
The data suggest that if $k$ is large, almost all numbers
of length $k$ containing $b-1$ as a digit
and not ending with zero are prime,
and so the number of primes of
length $k$ appears to approach 
$(b-1)^2\, b^{k-2}$ as $k \rightarrow \infty$ (Conjecture \ref{ConjPrime}). 
In any case, any number with a digit equal to $b-1$
is a product of primes (Theorem \ref{ThNinish}),
and {\em every} number can be written
as $r$ times a product of primes,
for some $r \in \{0, 1, \ldots, b-1\}$ (Corollary \ref{CorNinish}).
These factorizations are in general not unique.
There is a useful process using digit maps
for ``promoting'' a prime from a lower base to a higher base,
which enables us to replace the list
of all primes by a shorter list of prime
``templates'' (Table \ref{Tab71a}).

Dismal squares are briefly discussed in \S\ref{SecSquares}.

In \S\ref{SecPoset} we investigate the different 
ways to order the dismal numbers, and in 
particular the partially ordered set 
defined by the divisibility relation (see Table \ref{FigP}).
We will see that greatest common divisors
and least common multiples need not exist,
so this poset fails to be a lattice.
On the other hand, we do have the notion of ``relatively prime''
and we can define an analogue of the Euler totient function.

In \S\ref{SecDiv} we study the number-of-divisors function $d_b(n)$,
and investigate which numbers have the most divisors.
It appears that in any base $b \ge 3$, the number
$n = (b^k-1)/(b-1) = 111\ldots1|_b$ has more divisors than any other
number of length $k$. 
The binary case is slightly different.
Here it appears that among all $k$-digit numbers $n$,
the maximal value of $d_2(n)$ occurs at
$n = 2^k-2 = 111\ldots10|_2$, and this is the unique maximum
for $n \neq 2, 4$.
Among all {\em odd} $k$-digit numbers $n$,
$d_2(n)$ has a unique maximum at
$n = 2^k-1 = 111\ldots111|_2$, and if $k \ge 3$ and $k \ne 5$,
the next largest value occurs at 
$n = 2^k-3 = 111\ldots101|_2$, its reversal
$2^k-2^{k-2}-1 = 101\ldots111|_2$,
and possibly other values of $n$
(see Conjectures \ref{Conj1}-\ref{Conj3}).
Although we cannot prove these conjectures,
we are able to determine the exact values of
$d_b(111\ldots111|_b)$ 
and $d_2(111\ldots101|_2)$ (Theorem \ref{Th42},
which extends earlier work of
Richard Schroeppel and the second author, and Theorem \ref{Th62}).

The sequence of the number of divisors of $11\ldots11|_2$ (with $k$ $1$'s)
turns out to arise in a variety of different
problems, involving compositions, trees, polyominoes, Dyck paths, etc.---see
Remark (iii) following Theorem \ref{Th36}.
The initial terms can be seen in Table \ref{TabA079500}. This sequence
appears in two entries in \cite{OEIS}, A007059 and A079500,
and is the subject of a survey article by 
Frosini and Rinaldi \cite{FrRi06}.
The asymptotic behavior of this sequence was determined by Kemp \cite{Kem94} and
by Knopfmacher and Robbins \cite{KnRo05}, the latter using
the method of Mellin transforms---see \eqn{EqKemp}.
This is an example of an asymptotic expansion
where the leading term has an oscillating component which,
though small, does not go to zero.
It is amusing to note
that one of the first problems in which the asymptotic
behavior was shown to involve a nonvanishing oscillating term
was the analysis of the
average number of {\em carries}
when two $k$-digit numbers are added (Knuth \cite{Knu78},
answering a question of von Neumann; see also Pippenger \cite{Pip02}).
Here we see a similar phenomenon when there are no carries.
In studying Conjectures \ref{Conj2} and \ref{Conj3},
we observed that the numbers of divisors for the runners-up,
$2^k-3$ and $2^k-2^{k-2}-1$, appeared to be converging
to one-fifth of the number of divisors of $11\ldots11|_2$.
This is proved in Theorem \ref{ThM3}.
Our proof is modeled on Knopfmacher and Robbins's proof \cite{KnRo05}
of \eqn{EqKemp},
and we present the proof in such a way that it 
yields both results simultaneously.

The sum-of-divisors function  $\sigma_b(n)$ is the subject of \S\ref{SecSum}.
There are analogues of the perfect numbers, although
they seem not to be as interesting as in the classical case.
Section \ref{SecPart} discusses the dismal analogue of
the partition function. Theorems \ref{ThPart1} and \ref{ThPart2}
give explicit formulas for the number of partitions
of $n$ into distinct parts.

This is the second of a series of articles dealing with
various kinds of carryless arithmetic, and
contains a report of our investigations
into dismal arithmetic carried out during the period 2000--2011.
This work had its origin in a study by the second
author into the 
results of performing binary arithmetic calculations with the usual
addition and multiplication of binary digits replaced by other operations.
If addition and multiplication are replaced by
the logical operations OR and AND, respectively, 
we get base $2$ dismal arithmetic.
(If instead we use XOR and AND, the results are very different,
the squares for example now forming the Moser-de Bruijn sequence
A000695.) 
Generalizing from base $2$ to base $10$ and then to an arbitrary base led
to the present work.

In the first article in the series, \cite{Carry1},
addition and multiplication were carried out ``mod $10$'', with no carries.
A planned third part will discuss even more exotic arithmetics.

Although dismal arithmetic superficially resembles 
``tropical mathematics'' \cite{RST05},
where addition and multiplication are defined by
$x \oplus y := \min\{x,y\}$, $x \odot y := x+y$,
there is no real connection, since tropical mathematics is defined
over $\RR \cup \{\infty\}$,  uses carries, and is not base-dependent.

\noindent
\textbf{Notation.}
The base will be denoted by $b$ and the largest digit
in a base $b$ expansion by $\beta := b-1$.
We write $n = n_{k-1} n_{k-2} \ldots n_1 n_0|_b$ to denote
the base $b$ representation of the number 
$\sum_{i=0}^{k-1} n_i b^i$, and we define $\len_b(n) := k$.
The components $n_i$ will be called the digits of $n$, even if $b \neq 10$.
In the examples in this paper $b$ will be at most $10$, so the notation 
$n_{k-1} \ldots n_1 n_0|_b$ (without commas) is unambiguous.
The symbols $\CAsub _b$ and $\CMsub _b$ denote dismal addition and
multiplication, and we omit the base $b$ if it is
clear from the context.
All non-bold operators ($+$, $\times$, $<$, etc.) refer
to ordinary arithmetic operations, as do unqualified terms like 
``smallest,'' ``largest,'' etc.
We usually omit ordinary multiplication
signs, but never dismal multiplication signs.
We say that $p$ divides $n$ in base $b$ (written $p \prec _b n$) if
$p \CMsub _b \, q = n$ for some $q$, and that
$p = p_{k-1} \ldots p_1 p_0|_b$ is dominated by 
$n = n_{k-1} \ldots n_1 n_0|_b$ 
(written $p \ll _b n$) if $p_i \le n_i$ for all $i$.
The symbol ``$|_b$'' always marks the end of a base $b$
expansion of a number, and is never used for ``divides in base $b$.''

\section{Basic definitions and properties}\label{Sec2}

We began, as we all did, in base $10$, but from now on we will 
allow the base $b$ to be an arbitrary integer $\ge 2$.

Let $\sA$ denote the set of base $b$ ``digits'' $\{0, 1, 2, \ldots, b-1\}$,
equipped with the two binary operations
\beql{EqRules1}
m \CAsub _b\, n ~:=~ \max\{m, n\}, \quad
m \CMsub _b\, n ~:=~ \min\{m, n\}, \quad \mbox{~for~} m,n \in \sA \,.
\eeq
A {\em dismal number} is an element of
the semiring $\sA[X]$ of polynomials
$\sum_{i=0}^{k-1} n_i X^i$, $n_i \in \sA$.
If $M[X] := \sum_{i=0}^{k-1} m_i X^i$ and
$N[X] := \sum_{i=0}^{l-1} n_i X^i$ are dismal numbers
then their {\em dismal sum} is formed by taking the dismal
sum of corresponding pairs of digits, analogously to ordinary addition
of polynomials:
\beql{EqRules2}
M[X] \CAsub _b\, N[X] ~:=~  \sum_{i=0}^{\max\{k,l\}-1} p_i X^i \,,
\eeq
where $p_i := m_i \CAsub _b\, n_i$,
and their {\em dismal product} is
similarly formed by using dismal arithmetic 
to convolve the digits, analogously to ordinary multiplication
of polynomials:
\beql{EqRules3}
M[X] \CMsub _b\, N[X] ~:=~  \sum_{i=0}^{k+l-2} q_i X^i \,,
\eeq
where 
\begin{align}
& q_0 := m_0 \CMsub _b\, n_0 \,, \nonumber \\
& q_1 := (m_0 \CMsub _b\, n_1) ~\CAsub _b\,~ (m_1 \CMsub _b\, n_0) \,, \nonumber \\
& q_2 := (m_0 \CMsub _b\, n_2) ~\CAsub _b\,~ (m_1 \CMsub _b\, n_1) ~\CAsub _b\,~ (m_2 \CMsub _b\, n_0) \,, \nonumber \\
& \ldots \,. \nonumber
\end{align}
We will identify a dismal number $N(X) = \sum_{i=0}^{k-1} n_i X^i$
with the integer $n$ whose base~$b$ expansion is 
$n = \sum_{i=0}^{k-1} n_i b^i$
($n$ is obtained by evaluating the polynomial $N(X)$ at $X=b$),
and we define $\len_b(n) := k$.
The rules \eqn{EqRules1}-\eqn{EqRules3}
then translate into the rules for
dismal addition and multiplication stated in \S\ref{Sec1}:
there are no carries, and digits are combined
according to the rules in \eqn{EqRules1}.
The $b$-ary number $\sum_{i=0}^{k-1} n_i b^i$
will also be written as $n_{k-1} n_{k-2} \ldots n_1 n_0|_b$.
Dismal numbers are, by definition, always identified with nonnegative
integers. Note that $\len_b(m \CAsub _b\, n) = \max\{\len_b(m),
\len_b(n)\}$ and
$\len_b(m \CMsub _b\, n) = \len_b(m) + \len_b(n)-1$.

\begin{thm}\label{ThLaws}
The dismal operations $\CAsub _b$ and $\CMsub _b$
on $\sA[X]$ satisfy the commutative and associative laws,
and $\CMsub _b$ distributes over $\CAsub _b$.
\end{thm}

\begin{proof}
(Sketch.) Each law requires us to show the identity of two
polynomials, and so reduces to showing
that certain identities hold for the coefficients
of each individual degree in the two polynomials.
These identities are assertions about $\min$ and $\max$ in the set $\sA$,
which hold since $(\sA, \le)$ is a totally ordered set,
and $(\sA, \min, \max)$ is a distributive lattice (cf. \cite{Gra03}).
\end{proof}

If $R$ denotes the operation of reversing the order of digits
and $m$ and $n$ have the same length, then 
$R(m \CAsub _b\, n) = R(m) \CAsub _b\, R(n)$ and
$R(m \CMsub _b\, n) = R(m) \CMsub _b\, R(n)$.

Individual digits in a dismal sum or product can often be varied without
affecting the result, so dismal subtraction and division 
will not be defined.
Example: $16 \CAsub _{10} \, 75 = 26 \CAsub _{10} \, 75 = 76$,
$16 \CMsub _{10} \, 75 = 16 \CMsub _{10} \, 85 = 165$.
This is why dismal numbers form only a semiring.
On the other hand, this semiring does have a multiplicative identity (see the
next section), and there are no zero divisors.

In certain situations we can give a more precise
statement about how changing digits in a dismal sum or
product affects the result.
We begin with a lemma about ordinary functions of real variables.

\begin{lem}\label{Lem711}
Let $f$ be a single-valued function of real variables
$x_1, \ldots, x_k$, $k \ge 2$, formed by repeatedly composing
the functions 
$(x,y) \mapsto \min\{x,y\}$
and
$(x,y) \mapsto \max\{x,y\}$.
If $g$ is a nondecreasing function of $x$, meaning
that
\beql{EqNond}
x \le y \quad \Rightarrow \quad g(x) \le g(y)\,,
\eeq
then
\beql{Eq711}
f(g(x_1),\ldots,g(x_k)) ~=~ g(f(x_1,\ldots,x_k))\,,
\eeq
for all real $x_1,\ldots,x_k$.
\end{lem}

We omit the easy inductive proof.

We define a base $b$ {\em digit map} 
to be a nondecreasing function $g$ 
mapping $\{0,1,\ldots,b-1\}$ into itself.
The map $g$ need not be one-to-one or onto.
If $g$ is a digit map and $n=n_{k-1}\ldots n_1 n_0|_b$
then we set $g(n) := g(n_{k-1})\ldots g(n_1) g(n_0)|_b$.

\begin{thm}\label{Th712}
If $m$ and $n$ are dismal numbers and $g$ is a base $b$ digit map,
then
\begin{align}\label{Eq712}
g(m \CAsub _b\, n)  & ~=~ g(m)  \CAsub _b\, g(n) \,, \nonumber \\
g(m \CMsub _b\, n)  & ~=~ g(m)  \CMsub _b\, g(n) \,. 
\end{align}
\end{thm}

\begin{proof}
This follows from Lemma \ref{Lem711},
since the individual digits of $m \CAsub _b\, n$ and $m \CMsub _b\, n$
are functions of the digits of $m$ and $n$
of the type considered in that lemma.
\end{proof}

\begin{cor}\label{Cor705}
If $p=m \CMsub _b\, n$
then $p$ can also be written as 
$m' \CMsub _b\, n'$,
where $m'$ and $n'$ use only digits that are digits of $p$.
\end{cor}

\begin{proof} (Sketch.)
Arrange all distinct digits occurring in $p$, $m$, and $n$
in increasing order. Then construct a digit map $g$ by
increasing or decreasing the digits in $m$ and $n$
that are not in $p$ until they coincide
with digits of $p$,
leaving the digits of $p$ fixed. 
\end{proof}
For example, consider the product $165 = 16 \CMsub _{10} \, 85$
mentioned above.  The digits involved are $1, 5, 6, 8$,
and the digit map described in the proof fixes $1, 5,$ and $6$ and maps $8$ to $6$.
The resulting factorization is $165 = 16 \CMsub _{10} \, 65$.
(The additive analogue of Corollary \ref{Cor705} is true, but trivial.)

We will see other applications of Theorem \ref{Th712}
in the next section.

Note that when we are computing the base $b$ dismal sum
or product of two numbers $p$ and $q$, once
we have expressed $p$ and $q$ in base $b$, 
the value of $b$ plays no further role in the calculation.
Of course we need to know $b$ when we convert the result back to an integer,
but otherwise $b$ is not used.
So we have:
\begin{lem}
\label{LemmaP}
If the largest digit that is mentioned in a base $b$ dismal
sum or product is $d$ (where $0 \le d \le b-1)$,
then the same calculation is valid in any base that exceeds $d$.
\end{lem}

For example, here is the calculation of
the base $2$ dismal product of $13 = 1101|_2$ and
$5 = 101|_2$:
\begin{center} 
$$
\begin{tabular}{ p{.001in}  p{.001in} p{.001in} p{.001in} p{.001in} p{.001in} }
       &   & 1 & 1 & 0 & 1 \\
$\CMsub _2$  &   &   & 1 & 0 & 1 \\
\hline
       &   & 1 & 1 & 0 & 1 \\
     1 & 1 & 0 & 1 &   &   \\
\hline
     1 & 1 & 1 & 1 & 0 & 1 \\
\hline
\end{tabular}
$$
\end{center}
This tells us that $13 \CMsub_2 \, 5 = 61$,
but the same tableau can be read in base $3$, giving 
$37 \CMsub_3 \, 10 = 361$, or in base $10$,
giving $1101 \CMsub_{10} \, 101 = 111101$.

Recall that we say that $p$ {\em divides} $n$ in base $b$ (written $p \prec _b n$) if
$p \CMsub _b \, q = n$ for some $q$.
Since $\len_b(p) \le \len_b(n)$,
nonzero numbers have only finitely many divisors.
We also say that 
$p := p_{k-1} \ldots p_1 p_0|_b$ is {\em dominated by} 
$n := n_{k-1} \ldots n_1 n_0|_b$ 
(written $p \ll _b n$) if $p_i \le n_i$ for all $i$.
Then $p \ll _b n$ if and only if $p \CAsub _b \, n = n$.
Another consequence of Lemma \ref{Lem711} is:
\begin{lem}\label{LemDom}
If $p \ll _b m$ and $q \ll _b n$ then 
$p \CAsub _b\, q \ll _b m \CAsub _b\, n$ and
$p \CMsub _b\, q \ll _b m \CMsub _b\, n$.
\end{lem}

Finally, we remark without giving any details that, in
any base, the sets of
numbers with digits in nondecreasing order, or in
nonincreasing order (see A009994 and A009996 for base $10$)
are closed under dismal addition and multiplication.

\section{Dismal primes}\label{SecPrimes}

In dismal arithmetic in base $b$, for bases $b>2$,
the multiplicative identity is no longer $1$
(for example, $1 \CMsub _{10}\,23 = 11$, not $23$).
In fact, it follows from the definition of multiplication that
the multiplicative identity is the 
largest single-digit base $b$ number, $\beta := b-1$.
For base $b=10$ we have $\beta = 9$, 
and indeed the reader will easily check that 
$9 \CMsub _{10}\, n = n$ for all $n$.
An empty dismal product is defined to be $\beta$, by convention.

If $p \CMsub _b\, q = \beta$, then $p=q=\beta$,
so $\beta$ is the only unit.
We therefore define a {\em prime} in base~$b$ dismal arithmetic
to be a number, different from $\beta$, 
whose only factorization is $\beta$ times itself.

If $p$ is prime, then at least one digit of $p$
must equal $\beta$ (for if the largest
digit were $r<\beta$, then $p=r \CMsub_ b\, p$,
and $r$ would be a divisor of $p$).
The base $b$ expansions of the first few primes are
$1\beta$ (this is the smallest prime), $2\beta, 3\beta \ldots \beta-1\,\beta$,
$\beta 0, \beta 1, \ldots, \beta \beta, 10\beta, \ldots$.
In base $10$, the primes are
\begin{align}\label{EqPrimes}
& 19, \,  29, \,  39, \,  49, \,  59, \,  69, \,  79, \,  89, \,  90, \,  91, \,  92, \,  93, \,  94, \,  95, \,  96, \,  97, \,  98, \,  99, \,  109, \,  209, \,  219, \,   \nonumber \\
& 309, \,  319, \,  329, \,  409, \,  419, \,  429, \,  439, \,  509, \,  519, \,  529, \,  539, \,  549, \,  609, \,  619, \,  629, \,  639, \,  \ldots  
\end{align}
(A087097). Notice that the presence of a digit equal to $\beta$ is 
a necessary but not sufficient condition for a number to be a prime: 
$11\beta|_b = 1\beta|_b \CMsub _b\, 1\beta|_b$ is not prime
(see A087984 for these exceptions in the case $b=10$).
In base $2$, the primes (written in base $2$) are
\beql{EqPrimes2}
10, \, 11, \, 101, \, 1001, \, 1011, \, 1101, \, 10001, \, 10011, \, 10111, \, 11001, \, 11101, \, 100001, \, \ldots 
\eeq
(A171000).
In view of the interpretation of base $2$ dismal arithmetic
in terms of Boolean operations mentioned in \S\ref{Sec1}, 
the corresponding polynomials
$$ 
X,\, X+1,\, X^2+1,\, X^3+1,\, X^3+X+1,\, \
X^3+X^2+1, \,  X^4+1, \,  X^4+X+1, \, \ldots \,,
$$
together with $1$, might be called the OR-irreducible Boolean polynomials.
Their decimal equivalents, 
$$
1, \, 2, \,  3, \,  5, \,  9, \,  11, \,  13, \,  17, \,  19, \,  23, \,  25, \,  29, \,  33, \,  35, \,  37, \, \ldots
$$
form sequence A067139 in \cite{OEIS}, contributed by Jens Vo\ss\, in 2002.

All numbers of the form $100\ldots00\beta|_b$ (with zero or more internal
zeros) are base $b$ primes,
since there is no way that
$p \CMsub _b\, q$ can have the form $100\ldots00\beta|_b$
unless $p$ or $q$ is a single-digit number.
So there are certainly infinitely many primes in any base.

Since $1\beta|_b$ is the smallest prime, the numbers $1\beta|_b \CMsub _b\, n$ are 
another analogue of the even numbers. 
Whereas the first version of the even numbers,
given in \eqn{EqEven} for base $10$, was simply ``replace all digits in $n$  
that are bigger than $2$ with $2$'s,'' this version is more
interesting.
In base $10$ we get
\beql{EqA162672}
0, \, 11, \, 12, \, 13, \, 14, \, 15, \, 16, \, 17, \, 18, \, 19, \, 110, \, 111, \, 112, \, 113, \, 114, \, 115, \, 116, \, 117, \, 118, \, \ldots 
\eeq
(A162672, which contains repetitions and is not monotonic).

If $b>2$, there are numbers which cannot be written as a product of primes
(e.g., $1$).

\begin{thm}\label{ThNinish}
Any base $b$ number with a digit equal to $\beta$ is a (possibly empty) dismal product
of dismal primes.
\end{thm}

\begin{proof}
The number $\beta$ itself is the empty product of primes.
Every two-digit number with $\beta$ as a digit is already a prime.
If there are more than two digits, either the number is a prime,
or it factorizes into the product of two numbers,
both of which must have $\beta$ as a digit.
The result follows by induction.
\end{proof}

\begin{cor}\label{CorNinish}
Every base $b$ number can be written as $r$ times a dismal product
of dismal primes, for some $r \in \{0, 1, 2, \ldots, b-1\}$.
\end{cor}

\begin{proof}
Let $r$ be the largest digit of $n$.
If $r=\beta$ the result follows from the theorem.
Otherwise, let $m$ be
obtained by changing all occurrences of $r$
in the $b$-ary expansion of $n$ to $\beta$'s,
so that $n = r \CMsub _b\, m$,
and apply the theorem to $m$.
\end{proof}

Even when it exists, the
factorization into a dismal product of dismal primes is in general not unique.
In base $10$, for example,
the list of numbers with at least two different factorizations
into a product of primes is
\beql{EqA171004}
1119, \, 1129, \, 1139, \, 1149, \, 1159, \, 1169, \, 1179, \, 1189, \, 1191, \, 1192, \, 1193, \, 1194, \, 1195, \, \ldots \,,
\eeq
where for instance
$1119= 19 \CMsub _{10}\, 19 \CMsub _{10}\, 19 = 19 \CMsub _{10}\, 109$  (A171004).

\begin{table}[htbp]
$$
\begin{array}{|ccc|cc|cc|cc|} \hline
k & \mbox{base } 2 & \mbox{base } 10 & k & \mbox{base } 2 & k & \mbox{base } 2 & k & \mbox{base } 2  \\
\hline
1 & 0 & 0 & 11 & 323 & 21 & 442313 & 31 & 510471015 \\
2 & 2 & 18 & 12 & 682 & 22 & 902921 & 32 & 1027067090 \\
3 & 1 & 81 & 13 & 1424 & 23 & 1833029 & 33 & 2065390101 \\
4 & 3 & 1539 & 14 & 2902 & 24 & 3719745 & 34 & 4151081457 \\
5 & 5 & 20457 & 15 & 5956 & 25 & 7548521 & 35 & 8336751732 \\
6 & 9 & 242217 & 16 & 12368 & 26 & 15264350 & 36 & 16734781946 \\
7 & 19 & 2894799 & 17 & 25329 & 27 & 30859444 & 37 & 33583213577 \\
8 & 39 & 33535839 & 18 & 51866 & 28 & 62355854 & 38 & 67357328359 \\
9 & 77 & 381591711 & 19 & 106427 & 29 & 125773168 & 39 & 135056786787 \\
10 & 168 & ? &       20 & 217216 & 30 & 253461052 & 40 & ? \\
\hline
\end{array}
$$
\caption{Numbers of dismal primes with $k$
digits in bases $2$ and $10$ (A169912, A087636).}
\label{TabPT1}
\end{table}

We can, of course, study the primes dividing $n$,
even if $n$ does not contain a digit equal to $\beta$.
Without giving any details,
we mention that \cite{OEIS} contains
the following sequences:
the number of distinct prime divisors of $n$ (A088469),
their dismal sum (A088470),
and dismal product (A088471);\footnote{A088471 has an unusual beginning:
$9, 9, 9, 9, 9, 9, 9, 9, 9, 90, 123456789987654321, 19, 19, 19,  \ldots$ .}
also the lists of numbers $n$ such that the 
dismal sum of the distinct prime divisors of $n$ is
$<n$ (A088472),
$\le n$ (A088473),
$\ge n$ (A088475),
$> n$ (A088476);
as well as the numbers $n$ such that the 
dismal product of the distinct prime divisors of $n$ is
$<n$ (A088477),
$\le n$ (A088478),
$= n$ (A088479),
$\ge n$ (A088480), and
$> n$ (A088481).
There is no analogue of A088476 or A088481 in
ordinary arithmetic.

One omission from the above list is explained
by the following theorem.

\begin{thm}\label{Th11}
In base $b$ dismal arithmetic,
$n$ is prime if and only if the dismal sum
of its distinct dismal prime divisors is equal to $n$.  
\end{thm}
\begin{proof}
If $n=p$ is prime then the sum of the primes 
dividing it is $p$.
Suppose $n$ is not prime and let $m$ be the sum
of the distinct dismal primes diving $n$.
If $n$ is divisible by a prime $p$ with
$\len_b(p) = \len_b(n)$, then $n = r \CMsub _b\, p$, $r<\beta$,
the largest digit in $n$ is $r$, 
and so $m \ne n$ since $n \ll _b p \ll _b m$.
If $\len_b(p) < \len_b(n)$ for all prime divisors $p$,
then $\len_b(m) < \len_b(n)$, and again $m \ne n$. 
\end{proof}

We now consider how many primes there are.
Let $\pi_b(k)$ denote the number of base $b$ dismal primes
with $k$ digits.
Table \ref{TabPT1} shows the initial values of $\pi_2(k)$ and 
$\pi_{10}(k)$.
Necessary conditions for a number $n$ to be prime are that it
contain $\beta$ as a digit and (if $k>2$) does not end with $0$.
There are
\beql{EqNOP}
(b-1)^2\,b^{k-2} \,-\, (b-2)\,(b-1)^{k-2}
\eeq
such numbers.
It seems likely that, as $k$ increases,
almost all of these numbers will be prime,
and the data in Table \ref{TabPT1} is consistent with this. 
We therefore make the following conjecture.
\begin{conj}
\label{ConjPrime}
\beql{EqConjP}
\pi_k(b) ~\sim~ (b-1)^2\,b^{k-2} \quad \mbox{as } k \rightarrow \infty\,.
\eeq
\end{conj}

We can get a lower bound on $\pi_b(k)$ by producing large numbers
of primes, using the process of ``promotion.''
We call a base $b$ number with
at least two digits a {\em pseudoprime} if  its
only factorizations are of the form $n = p \CMsub _b\, q$ 
where at least one of $p$ and $q$ has length $1$.
In base $b$, $n$ is a prime if and only if
it is a pseudoprime and contains a digit $\beta$.
If $n_{k-1}n_{k-2}...n_0|_b$ is a pseudoprime and $r$ is its
maximal digit, then $n_{k-1}n_{k-2}...n_0|_{r+1}$ is a base $r+1$
prime and furthermore
$n_{k-1}n_{k-2}...n_0|_c$ is a pseudoprime in any base $c \ge r+1$.
In base $2$ there is no difference between primes
and pseudoprimes.
As long as we exclude numbers ending with $0$,
reversing the digits of a number does not change its
status as a prime or pseudoprime.

The advantage of working with pseudoprimes rather than primes is 
that the inverse image of a pseudoprime under a digit map (see \S\ref{Sec2})
is again a pseudoprime.

\begin{thm}\label{ThDM}
For a base $b$ digit map $g$, 
if $g(n_{k-1}n_{k-2}...n_0|_b)$ is a pseudoprime and $g(n_{k-1})$
is not $0$, then $n_{k-1}n_{k-2}...n_0|_b$ is a pseudoprime.
\end{thm}
\begin{proof}
This follows immediately from Theorem \ref{Th712}.
\end{proof}

So if $p$ is a pseudoprime, then any number $n$ with
the property that there is
a digit map sending $n$ to $p$ is also
a pseudoprime; we think of $n$
as being obtained by ``promoting'' $p$,
and call $p$ the ``template'' for $n$.

Here is an equivalent way to describe the 
promotion process. Suppose the distinct digits in $p$, the
template, or number to be
promoted, are $d_1 < d_2 < d_3 , \ldots$.
For each $d_i$, choose a set of digits $S(d_i)$
such that all the digits in $S(d_i)$ are strictly less than all those in $S(d_j)$,
for $i<j$.
Replace any digit $d_i$ in $p$ by any digit in $S(d_i)$.
All numbers $n$ obtained in this way are promoted
versions of $p$  (the required digit
map $g$ being defined by $g(c)=d_i$ for all $c \in S(d_i))$. 

Any pseudoprime (in any base) with at most four digits can be obtained
by promoting a base $2$ prime.
At length $2$, $11|_2$ is a prime (see \eqn{EqPrimes2}),
so every $2$-digit number $rs|_b$ is a pseudoprime,
using the digit map $g$ that sends $r$ and $s$ to $1$.
This is valid even if $s=0$, since \eqn{EqNond} still holds.
At length $3$, $101|_2$ is a base $2$ prime,
so any three-digit number $rst|_b$ with $r>s$, $t>s$
is a pseudoprime (take $g(r)=g(t)=1, g(s)=0$),
and this captures all three-digit pseudoprimes.
There are three templates of length $4$,
$1001|_2$, $1011|_2$, and $1101|_2$.
These can be promoted to capture all four-digit primes,
which are the numbers $rstu|_b$ for which one
of the following holds:
\begin{align}
& r \mbox {~and~} u \mbox{~are~both~strictly~greater~than~} s \mbox {~and~} t \,, \nonumber \\
& r, s, \mbox {~and~} u \mbox {~are~strictly~greater~than~} t \,, \nonumber \\
& r, t, \mbox {~and~} u \mbox{~are~strictly~greater~than~} s \,. \nonumber 
\end{align}
 
For lengths greater than four, we must use some nonbinary templates
to capture all pseudoprimes, and as the length $k$ increases
so does the fraction of nonbinary templates required,
as shown in Table \ref{Tab71}.
The columns labeled
(a) and (b) give the number of binary templates 
and the total number of templates, respectively,
and the columns (c) and (d) give the number
of base $10$ primes obtained by promoting the templates
in columns (a) and (b).

\begin{table}[htbp]
$$
\begin{array}{|r|r|r|r|r|} \hline
k & (a) & (b) & (c) & (d) \\
\hline
2 & 1 & 1 & 18 & 18 \\
3 & 1 & 1 & 81 & 81 \\
4 & 3 & 3 & 1539 & 1539 \\
5 & 5 & 8 & 17661 & 20457 \\
6 & 9 & 51 & 135489 & 242217 \\
\hline
\end{array}
$$
\caption{For lengths $k = 2$ through $6$,
the numbers of (a) binary templates, (b) all templates,
(c) base $10$ primes obtained by promoting the binary templates,
and (d) base $10$ primes obtained by promoting all the templates.}
\label{Tab71}
\end{table}

We can reduce the list of templates by omitting those that are reversals
of others.
Table \ref{Tab71a} shows the reduced list of templates
of lengths $\le 6$.

\begin{table}[htbp]
$$
\begin{array}{|r|r|r|r|} \hline
11 & 100001 & 102212 & 120212 \nonumber \\
101 & 100011 & 102221 & 120221 \nonumber \\
1001 & 100101 & 103223 & 120222 \nonumber \\
1011 & 100111 & 103233 & 121022 \nonumber \\
10001 & 101011 & 110212 & 121102 \nonumber \\
10011 & 101221 & 112021 & 122102 \nonumber \\
10111 & 101222 & 112022 & 122202 \nonumber \\
12021 & 102201 & 120021 & 132023 \nonumber \\
12022 & 102202 & 120022 & 133023 \nonumber \\
\hline
\end{array}
$$
\caption{Reduced list of prime templates: every pseudoprime of length $\le 6$
can be obtained by promoting one of these $36$
pseudoprimes or its reversal (A191420).}
\label{Tab71a}
\end{table}

Since $1\underbrace{00\ldots 0}_{k-2}1|_2$ is prime,
the promotion process tells us for example that
the numbers 
$r\underbrace{s_1s_2\ldots s_{k-2}}_{k-2}\beta|_b$ are prime,
for $1 \le r \le \beta$, provided each of the digits $s_i$ 
is in the range $0$ through $r-1$.  This gives 
$(b-1)^{k-2} + 2(b-2)^{k-2}+\cdots$ = $O(b^{k-2})$
primes of length $k$, which for $b>2$ is of exponential growth but 
smaller than \eqn{EqConjP}.

\section{Dismal squares}\label{SecSquares}
One might expect that it would be easier to find the number
of dismal squares of a given length than the number of dismal primes,
but we have not investigated squares as thoroughly, and 
we do not even have a precise conjecture about their asymptotic behavior. 
In base $2$, the first few dismal squares, written in base $10$,
are
\beql{EqSq2}
0, \, 1, \, 4, \, 7, \, 16, \, 21, \, 28, \, 31, \, 64, \, 73, \, 84, \, 95, \, 112, \, 125, \, 124, \, 127, \, 256, \, 273, \, \ldots \,,
\eeq
(A067398, also contributed by Jens Vo\ss\, in 2002),
and the numbers of squares of lengths $1$ (including $0$), 
$3, 5, 7, \ldots$ are
\beql{EqSq2l}
2, \, 2, \, 4, \, 8, \, 15, \, 29, \, 55, \, 105, \, 197, \, 367, \, 678, \, 1261, \, 2326, \, 4293, \, 7902, \, 14431, \, \ldots
\eeq
(A190820).
In base $10$, the first few squares were
given in \eqn{EqSquares}, and the numbers of squares
of lengths $1$ (including $0$), $3$, $5$, $7, \ldots$ are
$$
10, \,  90, \,  900, \,  9000, \,  74667, \, 608673, \ldots
$$
(A172199). 
The sequence of base $10$ squares is not monotonic
(for example $1011 < 1020$ yet
$1011 \CMsub _{10}\, 1011 = 1011111 > 1020 \CMsub _{10}\, 1020 = 1010200$),
and contains repetitions.
The numbers which are squares
in more than one way are
$$
111111111, \,  111111112, \,  111111113, \,  111111114, \,  111111115, \,  111111116, \,  111111117, \,  \ldots \,,
$$
e.g., $111111111 = 11011 \CMsub _{10}\, 11011 = 11111 \CMsub _{10}\, 11111$
(A180513, A181319).

We briefly mention two other questions
about squares to which we do not know the answer: 
(i) In base $2$, how many square roots does $2^{2k+1}-1$ have?
This is a kind of combinatorial covering problem.
For $k=0, 1, \ldots$ the counts are
\begin{align}\label{Eq191701}
~~~ & 1, \, 1, \, 1, \, 1, \, 2, \, 3, \, 5, \, 9, \, 15, \, 28, \, 50, \, 95, \, 174, \, 337, \, 637, \, 1231, \, 2373, \, 4618, \, 8974, \, 17567, \, 34387,  \nonumber \\ 
~~~ & 67561, \, 132945, \, 262096, \, 517373, \, 1023366, \, 2025627, \, 4014861, \, 7964971, \, 15814414, \nonumber \\
~~~ & 31424805, \, 62490481, \, 124330234, \, 247514283, \, 492990898, \, 982307460, \, 1958093809, \nonumber \\
~~~ & 3904594162, \, 7788271542, \, 15539347702, \, 31012331211, \, \ldots
\end{align}
(A191701). Is there a formula or recurrence for this sequence?
(ii) In base $b$, if we consider all $p$ such that
$p \CMsub_b \, p = n$, does one of them dominate all the others (in 
the $\gg_b$ sense)? If so 
the dominating one could be called the ``principal'' square root.

\section{The divisibility poset}\label{SecPoset}

One drawback to dismal arithmetic is that there is
more than one way to order the dismal numbers, and no ordering
is fully satisfactory.

The usual order on the nonnegative integers ($<$ or $\le$) is unsatisfactory,
since (working in base $10$) we have
$18 < 25 \mbox{ yet }  18 \CA 32 = 38 > 25 \CA 32 = 35$,
$32 < 41 \mbox{ yet }  32 \CM 3 = 32 > 41 \CM 3 = 31$.

The dominance order ($\ll_b$) is more satisfactory, in view
of Lemma \ref{LemDom} and the distributive law of 
Theorem \ref{ThLaws}, but has the drawback that $m$ divides $n$
does not imply that $m \ll _b n$ (e.g., $12|_{10}$ divides $11|_{10}$,
yet $11|_{10} \ll _{10} 12|_{10}$).

The partial order induced by divisibility
($\prec _b$) is worth discussing, as it
has some interesting properties and is the best way
to look at dismal numbers as long as we
are considering only questions of factorization and divisibility.
For simplicity we will restrict the discussion to base $10$.

\begin{figure}[htbp]
\begin{center}
\includegraphics[width=6.5in]{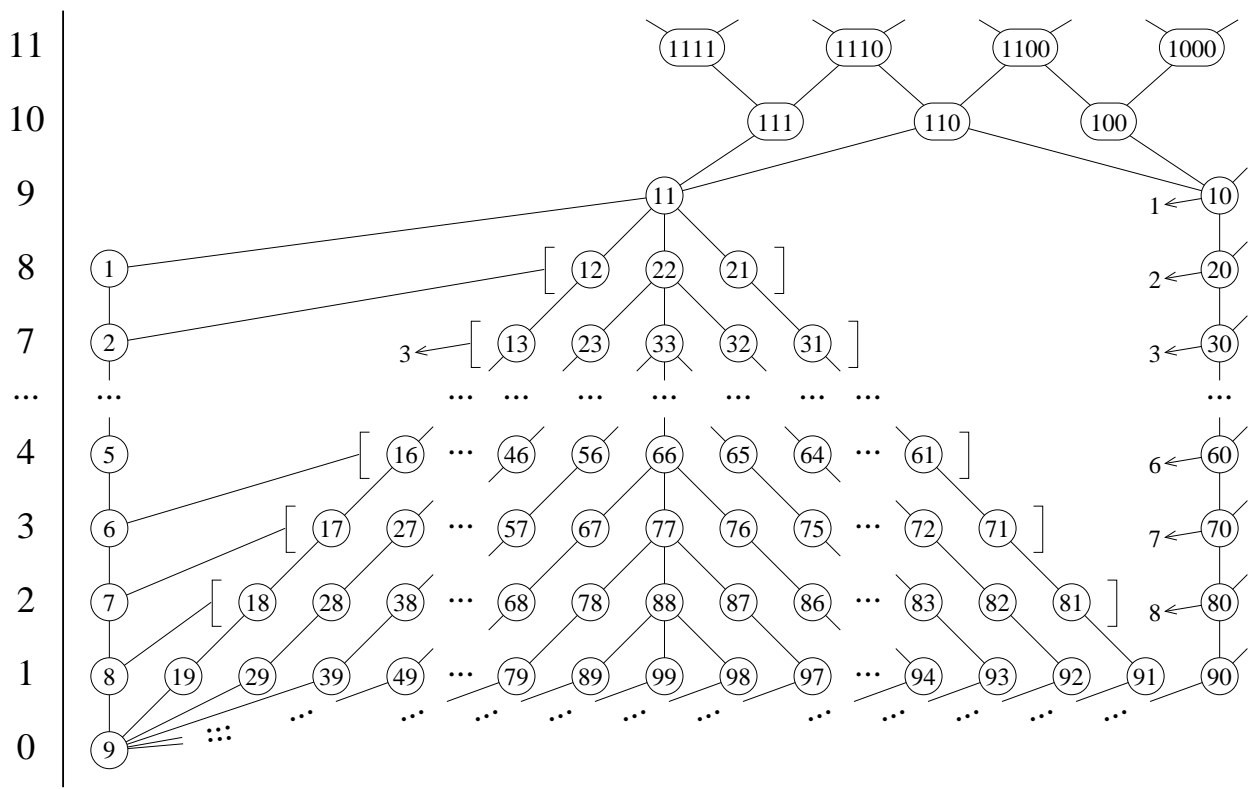}
\end{center}
\caption{Beginning of the divisibility poset (see text for description). The left-hand column gives the rank (A161813).}
\label{FigP}
\end{figure}

Figure \ref{FigP} displays the beginning of the Hasse diagram
(\cite[p.~99]{Stanley})
of this partially ordered set (or poset),
and shows all positive numbers with one or two digits,
and a few larger numbers. 
There are too many edges to draw in the diagram,
so we will describe them in words.

The multiplicative identity $9$, the ``zero element'' in
the poset, is at the base.
The other single-digit numbers $8 \prec 7 \prec 6 \prec \cdots \prec 1$
are above it in the left-hand column, and $8$ is joined to $9$.
The numbers are arranged in rows according to their
rank (shown at the extreme left of the diagram).
The numbers of rank $1$ consist of $8$ and the (infinitely many) primes:
$19, 29, \ldots, 91, 90, 109, 209, 219, 309, \ldots $ (A144171). 
All of these are joined to $9$.
The numbers of rank $2$ are $7, 18, 28, \ldots, 81, 80, 108, 119, \ldots$ (A144175),
and so on.
Every two-digit number of rank $h$ ($1 \le h \le 9$) is
joined to the single-digit number $h-1$ on the left
of the diagram., as indicated by the square brackets.
A two-digit number to the left of the central column of the pyramid
is joined to the number diagonally below it to the left
(e.g., $56$ is joined to $57$),
and a two-digit number to the right of the central column 
is joined to the number diagonally below it to the right
(e.g., $65$ is joined to $75$).
A number in the central column is joined to
the three numbers immediately below it (e.g., $66$ is joined to
$67, 77, 76$).
A number $r0$ ($1 \le r \le 9$)
in the right-hand column is joined to the number immediately
below it and to the single-digit number $r$.

Only a few numbers with more than two digits are shown, but
a more complete diagram
would show for example that $1$ is joined
to, besides $10$ and $11$, many other decimal numbers
whose digits are $0$'s and $1$'s, such as $101$, $1001$, $1011,\, \ldots$.
all of rank $9$.
The figure is complete in the sense that all downward joins
are shown for all the numbers in the diagram.

One perhaps surprising property of 
the divisibility poset is that the greatest lower bound (or 
greatest common divisor) $m \wedge n$ 
and the least upper bound (or least common multiple) $m \vee n$
of two dismal numbers $m$ and $n$
need not exist, and so this poset fails
to be a lattice \cite[Chap.~1]{Gra03}.
For example, again working in base $10$,
the rank $2$ numbers $8989$ and $9898$ are each divisible by
(and joined to) the nine primes $909, 919, \ldots, 989$.
However, these nine primes are incomparable
in the $\prec _{10}$ order, so neither 
$8989 \wedge 9898$ nor $909 \vee 919$ exist.

Although greatest common divisors need not exist, we 
can still define two dismal numbers to be
{\em relatively prime} if their only common divisor is the unit $\beta$. 

In the next section we will study the number of divisors function
$d_b(n)$. Candidates for divisors of $n$ are all numbers $m$ with
$\len_b(m) \le \len_b(n)$.
It is therefore appropriate to define the dismal analogue
of the Euler totient function, $\phi_b(n)$, to be
the number of numbers $m$ with $\len_b(m) \le \len_b(n)$
which are relatively prime to $n$.
The initial values of $\phi_2(n)$ and $\phi_{10}(n)$
are shown in Table \ref{TabEuler}.

\begin{table}[htb]
$$
\begin{array}{|c|rrrrrrrrrrrrrrrrrrrr|}
\hline
n          & 1 & 2 & 3 & 4 & 5 &  6 &  7 &  8 &  9 &  10 &  11 &  12 &  13 &   14 &   15 &   16   & 17 &  18 &  19 &  20 \\
\phi_2(n) & 1 &  2 &  2 &  4 &  6 &  2 &  4 &  8 &  14 &  6 &  14 &  5 &  14 &  5 &  7 &  16 & 30 &  14 &  30 &  12 \\
\phi_{10}(n) & 1 &  1 &  1 &  1 &  1 &  1 &  1 &  1 &  9 &  18 &  2 &  18 &  18 &  18 &  18 &  18 &  18 &  18 &  90 &  18 \\
\hline
\end{array}
$$
\caption{Values of totient functions $\phi_2(n)$ and $\phi_{10}(n)$
(A191674, A191675).}
\label{TabEuler}
\end{table}

\section{The number of dismal divisors}\label{SecDiv}

Let $d_b(n)$ denote the number of dismal divisors of $n$
in base $b$, and let $\sigma _b(n)$ denote the dismal sum
of the dismal divisors of $n$.
These functions are more irregular than their classical
analogues, as can be seen from the examples in 
Table \ref{Tab38}, and there are no simple formulas for them.
In this section we study some of the properties
of $d_b(n)$.
Note that if $r$ is the largest digit in $n$, then
the smallest divisor of $n$ is $r$, and
the largest divisor is the number obtained by changing all the $r$'s
in $n$ to $\beta$'s. 

\begin{table}[htbp]
$$
\begin{array}{|c|c|c|c|} \hline
n & \mbox{divisors~(base~10)} & d_{10}(n) & \sigma _{10}(n) \\
\hline
1 & 1, 2, 3, 4, 5, 6, 7, 8, 9 & 9 & 9 \\
2 & 2, 3, 4, 5, 6, 7, 8, 9    & 8 & 9 \\
3 & 3, 4, 5, 6, 7, 8, 9       & 7 & 9 \\
4 & 4, 5, 6, 7, 8, 9          & 6 & 9 \\
5 & 5, 6, 7, 8, 9             & 5 & 9 \\
6 & 6, 7, 8, 9                & 4 & 9 \\
7 & 7, 8, 9                   & 3 & 9 \\
8 & 8, 9                      & 2 & 9 \\
9 & 9                         & 1 & 9 \\
10 & 1, \ldots, 9, 10, 20, \ldots,  90 & 18 & 99 \\ 
11 & 1, \ldots, 9,  rs \mbox{~with~} 1 \le r,s \le 9 & 90 & 99 \\ 
12 & 2, \ldots, 9,  12, 13, \ldots, 19 & 16 & 19 \\ 
\hline
\end{array}
$$
\caption{In base $10$, the dismal divisors of 
the numbers $1$ through $12$ and the corresponding
values of $d_{10}(n)$ and $\sigma _{10}(n)$ (A189506, A087029, A087416).}
\label{Tab38}
\end{table}

A base $b$ dismal prime $p$ has two divisors, $b-1$ and $p$, so $d_b(p)=2$.
In the other direction, a divisor of a $k$-digit number $n$
has at most $k$ digits, so 
\beql{Eq1.1}
2 \le d_b(n) \le b^k - 1 \,.
\eeq
Base $b$ numbers of the form 
$111\ldots1|_b$ 
(that is, in which all the base $b$ digits are $1$) 
come close to meeting this upper bound---see Remark (iv) following
Theorem \ref{Th42}.
We make the following conjectures.

\begin{conj}
\label{Conj1}
In any base $b \ge 3$,
among all $k$-digit numbers $n$,
$d_b(n)$ has a unique maximum at 
$n = (b^k-1)/(b-1) = 111\ldots1|_b$. 
\end{conj}

\begin{conj}
\label{Conj2}
In base $2$, among all $k$-digit numbers $n$,
the maximal value of $d_2(n)$ occurs at
$n = 2^k-2 = 111\ldots10|_2$, and this is the unique maximum
for $n \neq 2, 4$.
\end{conj}

\begin{conj}
\label{Conj3}
In base $2$,
among all {\em odd} $k$-digit numbers $n$,
$d_2(n)$ has a unique maximum at 
$n = 2^k-1 = 111\ldots111|_2$, and if $k \ge 3$ and $k \ne 5$,
the second-largest value of $d_2(n)$ occurs at
$n = 2^k-3 = 111\ldots101|_2$, 
$n = 2^k-2^{k-2}-1 = 101\ldots111|_2$,
and possibly other values of $n$.
\end{conj}

The numerical evidence supporting these conjectures is compelling. 
For example, in base $10$, if we study the sequence 
$d_{10}(n)$, $n \ge 1$ (A087029) for $n \le 10^6$,
and write down $d_{10}(n)$ each time 
it exceeds $d_{10}(m)$ for all $m < n$, we obtain the values
$$
9, 18, 90, 180, 819, 1638, 7461, 14922, 67968 \,,
$$
(see A186443) at these (decimal) values of $n$:
$$
1, 10, 11, 110, 111, 1110, 1111, 11110, 11111 \,.
$$
If $n$ is a $5$-digit decimal number, the eight largest values of $d_{10}(n)$
are, in decreasing order,
$$
67968, 39624, 21812, 14922, 11202, 9616, 6732,  6570,
$$
at these values of $n$:
$$
11111, 22222, 33333, 11110, 44444, 12222 \mbox{~or~} 22220 \mbox{~or~} 22221, 11011, 10111 \mbox{~or~} 11101.
$$
The number $n = (10^k-1)/9 = 111\ldots1|_{10}$ is a clear winner
among all $k$-digit decimal numbers for $k \le 5$.
The data for bases $3$ through $9$ is equally supportive of Conjecture \ref{Conj1}.
Likewise, the binary data strongly supports Conjectures
\ref{Conj2} and \ref{Conj3}---see Table \ref{Tab40} for the
initial values of $d_2(n)$.
Table \ref{Tab40} also shows why $k=5$ is mentioned as an exception
in Conjecture \ref{Conj3}: among $5$-digit odd numbers,
$d_2(11011|_2)=4$ is the runner-up, ahead
of $d_2(10111|_2)=d_2(11101|_2)=2$.

\begin{table}[h]
$$
\begin{array}{|cc|cc|cc|cc|} \hline
n \mbox{~(base 2)} & d_{2}(n) & n \mbox{~(base 2)} & d_{2}(n) & n \mbox{~(base 2)} & d_{2}(n) & n \mbox{~(base 2)} & d_{2}(n) \\
\hline
-   & -   & 1000 & 4 & 10000 &  5 & 11000 & 8   \\
1   &   1 & 1001 & 2 & 10001 &  2 & 11001 & 2   \\
10  &   2 & 1010 & 4 & 10010 &  4 & 11010 & 4   \\
11  &   2 & 1011 & 2 & 10011 &  2 & 11011 & 4   \\
100 &   3 & 1100 & 6 & 10100 &  6 & 11100 & 9   \\
101 &   2 & 1101 & 2 & 10101 &  3 & 11101 & 2   \\
110 &   4 & 1110 & 6 & 10110 &  4 & 11110 & 10   \\
111 &   3 & 1111 & 5 & 10111 &  2 & 11111 & 8   \\
\hline
\end{array}
$$ 
\caption{In base $2$, the number of dismal divisors of 
the numbers $1$ through $31$ (A067399).}
\label{Tab40}
\end{table}

For a more dramatic illustration of Conjectures \ref{Conj1} and \ref{Conj2},
see the graphs of sequences A087029 and A067399 in \cite{OEIS}. 
Although the evidence is convincing, we
have not, unfortunately, succeeded in proving these conjectures.

We are able to determine the exact values of
$d_b(111\ldots111|_b)$ for all $b$ 
(the conjectural winner for $b \ge 3$ and the conjectural winner among odd numbers
in the binary case) and
$d_2(111\ldots101|_2) = d_2(101\ldots111)$ 
(the conjectural runners-up among odd binary numbers of length $k>5$).

We begin with a lemma that describes the effect of trailing zeros.

\begin{lem}
\label{LemmaZ}
If the base $b$ expansion of $n$ ends with exactly $r \ge 0$ zeros,
so that $n = m b^r$, with $b \nmid m$, then
$$
d_b(n) = (r+1) d_b(m) \,.
$$
\end{lem}

\begin{proof}
If $p \CMsub_b \, q = m$ then $b \nmid p$, $ b \nmid q$,
and the $r+1$ numbers $p b^i$ ($0 \le i \le r$) dismally divide $n$,
since
$$
(p b^i) \CMsub_b \, (q b^{r-i}) = n \,.
$$
Conversely, if $p' \CMsub_b \, q' = n$ then $p' = p b^i$, $q' = q b^{r-i}$,
$b \nmid p$, $ b \nmid q$, for some $i$ with $0 \le i \le r$.
So each dismal divisor of $m$ corresponds to
exactly $r+1$ dismal divisors of $n$.
\end{proof}

For example, in base $10$ the dismal divisors of $7$ are $7, 8, 9$
and the dismal divisors of $700$ are $7, 8, 9, 70, 80, 90, 700, 800, 900$.

One reason Conjectures \ref{Conj1}-\ref{Conj3} seem hard to prove
is the erratic behavior of $d_b(n)$. In contrast to the above lemma,
the effect of internal zeros is hard to analyze. Suppose the $i$-th
digit in the $b$-ary expansion of $n$ is zero.
This implies that if $n = p \CMsub_b \, q$, then all entries in the $i$-th column
of the long multiplication tableau must be zero,
which imposes many constraints on the $b$-ary expansions
of $p$ and $q$. One would expect, therefore, that changing the 
zero digit to a larger number---thus weakening the constraints---would 
always increase the number of divisors of $n$. 
Roughly speaking, this is true, but there are 
many cases where it fails.
For example, $d_2(11111|_2)=8$, but $d_2(11110|_2)=10$ (see Table \ref{Tab40},
Lemma \ref{LemmaZ} and Theorem \ref{Th36}). 
Again, $d_2(10101|_2)=3$, but $10111|_2$ is prime, so $d_2(10111|_2)=2$.
In any base $b$, 
$11\underbrace{00\ldots0}_{r}|_b$
has $(r+1)((b-1)^2+(b-1))$ divisors,
whereas 
$11\underbrace{00\ldots0}_{r-1}1|_b$
has $(b-1)^3+(b-1)$ divisors, a smaller
number if $r$ is large.

The next result was conjectured by the second author and proved
by Richard Schroeppel in 2001 \cite{LeSc01}.
An alternative proof (via a bijection with a certain class
of polyominoes) was given by Frosini and Rinaldi in 2006 \cite{FrRi06}.
We give a version of Schroeppel's elegant direct proof,
partly because it has never been published, and
partly because we will use similar arguments later. 

\begin{thm}\label{Th36}
In base $2$, the number of dismal divisors of $111\ldots1|_2$ $($with $k$ $1$'s$)$
is equal to the number of compositions of $k$ into parts of which
the first is at least as great as all the other parts.
\end{thm}

\begin{proof}
Suppose $p \CMsub_2 \, q = 111\ldots1|_2$  (with $k$ $1$'s)
where $\len_2(p)=r$, $\len_2(q)=k+1-r$.
By examining the long multiplication tableau for $p \CMsub_2 \, q$,
we see that it is also true that $p \CMsub_2 \, q' = 111\ldots1|_2$ where
$q' = 111\ldots1|_2$,  with $k+1-r$ $1$'s
(for if there is a $1$ in each column of the 
tableau for $p \CMsub_2 \, q$, that is still true for $p \CMsub_2 \, q'$).
So in order to find all the divisors $p$ of $111\ldots1|_2$ (with $k$ $1$'s)
we may assume that the cofactor $q$ has the form $111\ldots1|_2$
(with $s$ $1$'s, for some $s$, $1 \le s \le k$).

We establish the desired result by exhibiting a bijection between the two sets.
Let $k = c_1 + c_2 + \ldots + c_t$ be a composition
of $k$ in which $c_1 \ge c_i \ge 1$ ($2 \le i \le t$).
Let $\psi(c_i)$ denote the binary vector
$000\ldots 01$ with $c_i-1$ $0$'s and a single $1$.
The divisor $p$ corresponding to this composition
has binary representation given by the concatenation
\beql{Eqbij1}
1\,\psi(c_2)\,\psi(c_3)\,\ldots\,\psi(c_r)|_2 \,,
\eeq
of length $k+1-c_1$. If we set $q = 111\ldots1|_2$, of length $c_1$,
then $p \CMsub_2 \, q = 111\ldots1|_2$  (with $k$ $1$'s).
This follows from the fact that if the binary
representation of $p$ contains a string of exactly $s$ $0$'s:
$$
\ldots 1\underbrace{000\ldots0}_{s}1\ldots \,,
$$
then, when we form the product $p \CMsub_2 \, q$, 
the $1$ immediately to the right of these $0$'s will
propagate leftward to cover the $0$'s if and only if
$\len_2(q) \ge s+1$, which is exactly the condition that
$c_1 \ge c_i$ for all $i \ge 2$.
\end{proof}

For example, the eight compositions of $5$ 
in which no part exceeds the first
and the corresponding factorizations $p \CMsub_2 \, q$ of $11111|_2$ are
shown in Table \ref{Tab36}. 
Dots have been inserted in $p$ to indicate the division into the pieces $\psi(c_i)$.

\begin{table}[htbp]
$$
\begin{array}{|l|l|l|} \hline
\mbox{Composition of 5} & \mbox{divisor}~p & \mbox{cofactor}~q \\
\hline
5 & {1}|_2 & 11111|_2 \\
41 & 1.1|_2 & 1111|_2 \\
32 & 1.01|_2 & 111|_2 \\
311 & 1.1.1|_2 & 111|_2 \\
221 & 1.01.1|_2 & 11|_2 \\
212 & 1.1.01|_2 & 11|_2 \\
2111 & 1.1.1.1|_2 & 11|_2 \\
11111 & 1.1.1.1.1|_2 & 1|_2 \\
\hline
\end{array}
$$
\caption{Illustrating the bijection used to prove Theorem \ref{Th36}.}
\label{Tab36}
\end{table}

\noindent
\textbf{Remarks.}

(i) Using the bijection defined by \eqn{Eqbij1}, the
number of $1$'s in the binary expansion of the divisor $p$
is equal to the number of parts in the corresponding composition.

(ii) It follows immediately from the interpretation in terms
of compositions that the numbers $d_2(2^k-1)$ have generating function
\begin{align}\label{Eq3}
\sum_{k = 1}^{\infty} d_2(2^k-1) \, z^k & ~=~  
\sum_{l = 1}^{\infty} \frac{z^{l}}{1-(z+z^2+\cdots + z^{l})} \nonumber \\
        {} & ~=~  
\sum_{l = 1}^{\infty} \frac{(1-z) z^{l}}{1-2z+z^{l+1}}
\end{align}
(the index of summation, $l$, corresponds to the first part 
in the composition).

(iii) The initial values of this sequence are shown in Table \ref{TabA079500}.
\begin{table}[htb]
$$
\begin{array}{|c|rrrrrrrrrrrrrrrr|}
\hline
k          & 1 & 2 & 3 & 4 & 5 &  6 &  7 &  8 &  9 &  10 &  11 &  12 &  13 &   14 &   15 &   16  \\
d_2(2^k-1) & 1 & 2 & 3 & 5 & 8 & 14 & 24 & 43 & 77 & 140 & 256 & 472 & 874 & 1628 & 3045 & 5719  \\
\hline
\end{array}
$$
\caption{Values of $d_2(2^k-1)$.}
\label{TabA079500}
\end{table}
This sequence appears in entries A007059 and A079500 in \cite{OEIS},
although the indexing is different in each case.
The sequence also occurs in at least four other contexts besides the two
mentioned in Theorem \ref{Th36}, namely in the enumeration
of balanced ordered trees (Kemp \cite{Kem94}),
of polyominoes that tile the plane by translation
(Beauquier and Nivat \cite{BeNi91},
Brlek et al. \cite{BFRV06}),
of Dyck paths (see A007059),
and in counting solutions to the postage stamp problem
(again see A007059).
The article by Frosini and Rinaldi \cite{FrRi06} gives 
bijections between four of these six enumerations.
In this context we should also mention the recent article
of Rawlings and Tiefenbruck \cite{RaTi10},
which, although not directly related to the problems we consider,
discusses other connections between the enumeration 
of compositions, permutations, polyominoes, and binary words.

(iv) The asymptotic behavior of this sequence is quite subtle.
From the work of Kemp \cite{Kem94} and 
Knopfmacher and Robbins \cite{KnRo05} 
it follows that 
\beql{EqKemp}
d_2(2^k-1) ~\sim~ \frac{2^{k}}{k \log 2}~(1 + \Theta_k), \mbox{~as~} k \rightarrow \infty \,,
\eeq
where $\Theta_k$ is a bounded oscillating function
with $|\Theta_k| < 10^{-5}$
(see the proof of Theorem \ref{ThM3} below).

\vspace*{+.1in}

In order to determine $d_b(111\ldots1|_b)$ for bases $b>2$,
we first classify compositions 
in which no part exceeds the first
according to the number of parts. 
Let $T(k,t)$ denote the number of 
compositions of $k$ into exactly $t$ parts (with $1 \le t \le k$)
such that no part exceeds the first.
Table \ref{Tab41} shows the initial values.
This is entry A184957 in \cite{OEIS}.\footnote{An array equivalent to this, A156041,
was contributed to \cite{OEIS} by J.~Grahl in 2009
and later studied by A.~P.~Heinz and R.~H.~Hardin.}

\begin{table}[htbp]
$$
\begin{array}{|c|cccccccc|} \hline
k \backslash t & 1 & 2 & 3 & 4 & 5 & 6 & 7 & 8 \\
	  \hline
1 & 1 &   &   &   &   &   &   &   \\
2 & 1 & 1 &   &   &   &   &   &   \\
3 & 1 & 1 & 1 &   &   &   &   &   \\
4 & 1 & 2 & 1 & 1 &   &   &   &   \\
5 & 1 & 2 & 3 & 1 & 1 &   &   &   \\
6 & 1 & 3 & 4 & 4 & 1 & 1 &   &   \\
7 & 1 & 3 & 6 & 7 & 5 & 1 & 1 &   \\
8 & 1 & 4 & 8 & 11 & 11 & 6 & 1 & 1  \\
\hline
\end{array}
$$
\caption{Initial values of of $T(k,t)$, the number of
compositions of $k$ into exactly $t$ parts
such that no part exceeds the first.}
\label{Tab41}
\end{table}

The values of $T(k,t)$ are easily computed via the
auxiliary variables $\gamma(k,t,m)$, which we define to be
the number of compositions of $k$ into $t$ parts of
which the first part, $m$, is the greatest
(for $1 \le t \le k$, $1 \le m \le k$). 
We have the recurrence
\beql{Eq5}
\gamma(k,t,m) ~=~ \sum_{j=1}^{\min\{m, k+2-t-m\}} \gamma(k-j, t-1,m) 
\eeq
(classifying compositions according to the {\em last} part, $j$),
for $m>1$, $t>1$, $t+m<k-1$, with initial conditions
\begin{align}
\gamma(k,t,1) & ~=~ \delta_{t,k} \,,  \nonumber \\
\gamma(k,1,m) & ~=~ \delta_{m,k} \,,  \nonumber 
\end{align}
where $\delta_{i,j}=1$ if $i=j$ or $0$ if $i \ne j$.
Then
\beql{Eq6}
T(k,t) ~=~ \sum_{m=1}^{k+1-t} \gamma(k,t,m) \,.
\eeq
Since $\gamma(k,t,m)$ is the coefficient of $z^k$ 
in $z^m(z+z^2+\cdots+z^m)^{t-1}$, 
it follows that column $t$ of Table \ref{Tab41} has
generating function
\beql{Eq7}
\sum_{k=1}^{\infty} T(k,t) z^k ~=~
\frac{z^{t-1}}{(1-z)^{t-1}} ~ 
\sum_{r=0}^{t-1} (-1)^r \binom{t-1}{r} \frac{z^{r+1}}{1-z^{r+1}} \,.
\eeq

Since the total number of compositions of $k$ into $t$ parts
is $\binom{k-1}{t-1}$, and in at least a fraction $\frac{1}{t}$
of them the first part is the greatest, we have the bounds
\beql{Eqstars}
\frac{1}{t} \, \binom{k-1}{t-1} ~\le~ T(k,t) ~\le~ \binom{k-1}{t-1} \,.
\eeq

\begin{thm}\label{Th42}
\beql{Eq42}
d_b \Big( \frac{b^k-1}{b-1}\Big) ~=~ d_b(\underbrace{11\ldots 1}_{k}|_b) ~=~ 
\sum_{t=1}^{k} T(k,t) (b-1)^t \,.
\eeq
\end{thm}
\begin{proof}
Suppose $p \CMsub _b\, q = \underbrace{11\ldots 1}_{k}|_b$.
At least one of $p$ and $q$, say $q$,
must contain only digits $0$ and $1$ (for if $p$ contains a digit $i>1$
and $q$ contains a digit $j>1$,
then $i \CMsub _b\, j = \min\{i,j\} > 1$
will appear somewhere in $p \CMsub _b\, q$).
As in the proof of Theorem \ref{Th36} we may assume that this $q$
has the form $\underbrace{11\ldots 1}_{s}|_b$ for some $s$ with $1 \le s \le k$.
Suppose $p = \sum_{i=0}^{r-1} p_i 2^i$ with $p_i \in \{0,1\}$
is a divisor of $\underbrace{11\ldots 1}_{k}|_2$, so that
\beql{Eq42.2}
p \CMsub _2 \, (2^{k+1-r} -1) ~=~ 2^k-1 \,.
\eeq
By Lemma \ref{LemmaP}, $p' := \sum_{i=0}^{r-1} p_i b^i$
is a base $b$ dismal divisor of $\frac{b^k-1}{b-1}$\,:
\beql{Eq42.3}
p' \CMsub _b \, \frac{b^{k+1-r} -1}{b-1} ~=~ \frac{b^k-1}{b-1} \,.
\eeq
Furthermore, \eqn{Eq42.3} still holds if any of the $p_i$ that are $1$ are
changed to any digit in the range $\{1, 2, \ldots, b-1\}$.
Conversely, any base $b$ dismal divisor $p'$ of $\frac{b^k-1}{b-1}$
remains a divisor if all the nonzero digits
in the base $b$ expansion of $p'$  are replaced by $1$'s.
So each divisor of $\underbrace{11\ldots 1}_{k}|_2$ with $t$ $1$'s corresponds to $(b-1)^t$ divisors
of $\frac{b^k-1}{b-1}$.
Since there are $T(k,t)$ divisors of $2^k-1$ with $t$ $1$'s,
the result follows. 
\end{proof}

\begin{table}[htbp]
$$
\begin{array}{|c|rrrrrrrrr|} \hline
k \backslash b & 2 & 3 & 4 & 5 & 6 & 7 & 8 & 9 & 10  \\
\hline
1 & 1 & 2 & 3 & 4 & 5 & 6 & 7 & 8 & 9  \\
2 & 2 & 6 & 12 & 20 & 30 & 42 & 56 & 72 & 90 \\
3 & 3 & 14 & 39 & 84 & 155 & 258 & 399 & 584 & 819 \\
4 & 5 & 34 & 129 & 356 & 805 & 1590 & 2849 & 4744 & 7461 \\
5 & 8 & 82 & 426 & 1508 & 4180 & 9798 & 20342 & 38536 & 67968 \\
6 & 14 & 206 & 1434 & 6452 & 21830 & 60594 & 145586 & 313544 & 619902 \\
7 & 24 & 526 & 4890 & 27828 & 114580 & 375954 & 1044246 & 2555080 & 5660208 \\
\hline
\end{array}
$$
\caption{Table of 
$d_b(11\ldots 1|_b)$ (with $k$ $1$'s), the number of
base $b$ dismal divisors of $\frac{b^k-1}{b-1}$
(rows: A002378, A027444, A186636; columns A079500, A186523).}
\label{Tab41a}
\end{table}

\noindent
\textbf{Remarks.}

(i) Table \ref{Tab41a} shows the initial values of $d_b(\underbrace{11\ldots 1}_{k}|_b)$.

(ii) Theorem \ref{Th42} reduces to 
Theorem \ref{Th36} in the case $b=2$.

(iii)  From \eqn{Eqstars} and \eqn{Eq42} we have
\beql{Eqstars2}
\frac{b^k-1}{k} ~\le~ d_b(\underbrace{11\ldots 1}_{k}|_b) ~\le~ (b-1)b^{k-1} \,.
\eeq
For $b=2$ we also have the asymptotic estimate \eqn{EqKemp}.

\vspace*{+.2in}

We now study the runners-up in the binary case
(among odd numbers of length greater than $5$), namely the 
numbers $111\ldots101|_2$ and $101\ldots111|_2$. 
The simplest way to state the result is to give the generating
function.
\begin{thm}\label{Th62}
\begin{align}\label{Eq62}
\sum_{k = 3}^{\infty} d_2(2^k-3) z^k & ~=~
z ~+~ \frac{z^3}{1-z} ~+~ \sum_{l=3}^{\infty} \frac{(1-z)^2 z^l}{1-2z+z^{l-1}-z^l+z^{l+2}} \nonumber \\
        {} & ~=~
z+2z^3+2z^4+2z^5+4z^6+6z^7+10z^8+\cdots \,. \quad \rm{(A188288)}
\end{align}
\end{thm}

\noindent
We will deduce Theorem \ref{Th62} from Theorem \ref{ThM2} below.

Suppose $p \CMsub _2\, q = 2^k-3$, $k \ge 3$,
where $\len_2(p)=h$, $\len_2(q)=l$, with
$h+l=k+1$.
In order to find all choices for $p$, we note
that the binary expansions of $p$ and $q$ must end with $\ldots01$,
and that, as in the proofs of Theorems \ref{Th36} and
\ref{Th42}, we may assume that $q=2^l-3$.
Our approach is to fix $l$ and allow $h$ to vary.
Let $M_{h}^{(l)}$ denote the number of binary numbers $p$
with $\len_2(p)=h$ such that
\beql{EqP14a}
p \CMsub _2\, \underbrace{111\ldots101}_{l}|_2
~=~ \underbrace{111\ldots101}_{h+l-1}|_2 \,.
\eeq
Suppose the binary expansion of $p$ is
$$
1\,x_v\,x_{v-1}\,\ldots\,x_3\,x_2\,x_1\,0\,1|_2\,,
$$
where $v:=h-3$ and the $x_i$ are $0$ or $1$.
The long multiplication tableau for \eqn{EqP14a} 
implies that the $x_i$ must satisfy certain 
Boolean equations (remember that $\CAsub _2$ is the logical OR;
in what follows we will write $\CA$ rather than $\CAsub _2$).
For example, the tableau for $l=4$ and $h=9$, $v=6$ is shown in
Figure 2. 

\begin{center}
$$
\begin{tabular}{ p{.0016in} p{.0016in} p{.0016in} p{.0016in} p{.0016in} p{.0016in} p{.0016in} p{.0016in} p{.0016in} p{.0016in} p{.0016in} p{.0016in} }
       &   &  & 1 & $x_6$ & $x_5$ & $x_4$ & $x_3$ & $x_2$ & $x_1$ & 0 & 1 \\
$\CMsub_2 $  &   &  &   &     &     &     &     &   1 &   1 & 0 & 1 \\
\hline
       &   &  & 1 & $x_6$ & $x_5$ & $x_4$ & $x_3$ & $x_2$ & $x_1$ & 0 & 1 \\
    & 1 & $x_6$ & $x_5$ & $x_4$ & $x_3$ & $x_2$ & $x_1$ & 0 &   1 &   &   \\
  1 & $x_6$ & $x_5$ & $x_4$ & $x_3$ & $x_2$ & $x_1$ & 0 &   1 &   & &   \\
\hline
1      & 1 & 1& 1 &   1 &   1 &   1 &   1 &   1 &   1 & 0 & 1 \\
\hline
\end{tabular}
$$
Fig. 2.
\end{center}

\noindent
By reading down the columns, we obtain the equations
\begin{align}
x_1 \CA x_3 & ~=~ 1 \,, \nonumber \\
x_1 \CA x_2 \CA x_4 & ~=~ 1 \,, \nonumber \\
x_2 \CA x_3 \CA x_5 & ~=~ 1 \,, \nonumber \\
x_3 \CA x_4 \CA x_6 & ~=~ 1 \,, \nonumber \\
x_5 \CA x_6 & ~=~ 1 \,. \nonumber
\end{align}
There are $M_{9}^{(4)} = 29$ solutions $(x_1, \ldots, x_6)$ to
these equations. Table \ref{Tab61} shows the initial values of $M_{h}^{(l)}$,
as found by computer.

\begin{table}[htbp]
$$
\begin{array}{|c|cccccccc|} \hline
h \backslash l 
   & 1 & 2 & 3 & 4 & 5 & 6 & 7 & 8 \\
\hline
 1 & 1 & 0 & 1 & 1 & 1 & 1 & 1 & 1 \\
 2 & 0 & 0 & 0 & 0 & 0 & 0 & 0 & 0 \\
 3 & 1 & 0 & 0 & 1 & 1 & 1 & 1 & 1 \\
 4 & 1 & 0 & 1 & 1 & 2 & 2 & 2 & 2 \\
 5 & 1 & 0 & 2 & 3 & 3 & 4 & 4 & 4 \\
 6 & 1 & 0 & 2 & 5 & 7 & 7 & 8 & 8 \\
 7 & 1 & 0 & 3 & 9 & 13 & 15 & 15 & 16 \\
 8 & 1 & 0 & 6 & 16 & 24 & 29 & 31 & 31  \\
 9 & 1 & 0 &10 & 29 & 47 & 56 & 61 & 63  \\
10 & 1 & 0 &15 & 53 & 89 & 110 & 120 & 125  \\
11 & 1 & 0 &24 & 96 & 170 & 216 & 238 & 248  \\
12 & 1 & 0 &40 & 174 & 326 & 422 & 471 & 494  \\
\hline
\end{array}
$$
\caption{Table of $M_{h}^{(l)}$ (columns 3 and 4 are A070550 and A188223).}
\label{Tab61}
\end{table}

Inspection of the table suggests that the $l$-th column
satisfies the recurrence
\beql{EqMrec}
M_{h}^{(l)} ~=~
M_{h-1}^{(l)} +
M_{h-2}^{(l)} + \cdots +
M_{h-l+2}^{(l)} +
M_{h-l}^{(l)} +
M_{h-l-1}^{(l)} \,,
\eeq
for $l \ge 3$.
This will be established in Corollary \ref{Cor7}.

We consider the cases $h \le l+1$ and $h \ge l+2$ separately.
For $h \le l+1$, it is straightforward
to show the following:
\beql{EqRT0}
M_2^{(1)}~=~0 \,, \quad  
M_{h}^{(1)} ~=~ 1~(h \ne 2) \,, \quad M_h^{(2)} ~=~ 0 \,,
\eeq
and, for $l \ge 3$, $h \le l+1$,
\beql{EqRT1}
M_{h}^{(l)} ~=~
\begin{cases}
1,  &\text{if $h=1$}, \\
0,  &\text{if $h=2$}, \\
2^{h-3},  &\text{if $3 \le h \le l-1$}, \\
2^{h-3}-1,  &\text{if $h=l$ or $l+1$}.  \\
\end{cases}
\eeq
This accounts for the entries
in Table \ref{Tab61} that are on or above the line $h-l=1$. 

We now consider the case $3 \le l \le h-2 = v+1$.
The multiplication tableau leads to two special equations,
\beql{EqRT2}
\begin{cases}
x_2 ~=~ 1,  &\text{if $l=3$, or} \\
x_{1} \CA x_2 \CA \cdots \CA x_{l-3} \CA x_{l-1} ~=~ 1,  &\text{if $l \ge 4$}, \\
\end{cases}
\eeq
and
\beql{EqRT3}
x_{v-l+3} \CA x_{v-l+4} \CA \cdots  \CA  x_{v-1} \CA x_v ~=~ 1 \,,
\eeq
together with a family of $v-l+1$ further equations, which, if $l=3$, are
\beql{Eql3}
x_1 \CA x_3 ~=~ x_2 \CA x_4 ~=~ \cdots ~=~ x_{v-3} \CA x_{v-1} ~=~ x_{v-2} \CA x_{v} ~=~ 1 \,, 
\eeq
or, if $l \ge 4$, are
\begin{align}\label{EqRT4}
x_1  \CA  x_2  \CA  x_3  \CA  \cdots  \CA  x_{l-2}  \CA  x_{l}     & ~=~ 1 \,, \nonumber \\
x_2  \CA  x_3  \CA  x_4  \CA  \cdots  \CA  x_{l-1}  \CA  x_{l+1}   & ~=~ 1 \,, \nonumber \\
 \ldots \ldots   \ldots \ldots & \ldots \nonumber \\
x_{v-l+1}  \CA  x_{v-l+2}  \CA  x_{v-l+3}  \CA  \cdots  \CA  x_{v-2}  \CA  x_{v}   & ~=~ 1 \,. \nonumber \\
\end{align}

The two special equations \eqn{EqRT2} and \eqn{EqRT3}
involve variables with both low and high indices,
which makes induction difficult. We therefore define a 
simpler system of Boolean equations in which the special constraints
apply only to the high-indexed variables. 

For $l \ge 3$ and $n \ge 1$, let $D_{n}^{(l)}$ denote the number of
binary vectors $x_1 x_2 \ldots x_n$ of length $n$ that end with $x_n = 1$, 
do not contain any substring
$$
\underbrace{00\ldots000}_{l}
\mbox{~or~}
\underbrace{00\ldots010}_{l}
$$
and do not end with
$$
\underbrace{00\ldots01}_{l-1} \,.
$$

\noindent
Equivalently, $D_{n}^{(l)}$ is the number of solutions to the 
Boolean equations
\begin{align}\label{EqRT5}
x_1  \CA  x_2  \CA  x_3  \CA  \cdots  \CA  x_{l-2}  \CA  x_{l}     & ~=~ 1 \,, \nonumber \\
x_2  \CA  x_3  \CA  x_4  \CA  \cdots  \CA  x_{l-1}  \CA  x_{l+1}   & ~=~ 1 \,, \nonumber \\
 \ldots \ldots   \ldots \ldots & \ldots \nonumber \\
x_{n-l}  \CA  x_{n-l+1}  \CA  x_{n-l+2}  \CA  \cdots  \CA  x_{n-3}  \CA  x_{n-1}   & ~=~ 1 \,,
\end{align}
and
\beql{EqRT6}
x_{n-l+2} \CA x_{n-l+3} \CA \cdots  \CA  x_{n-2} \CA x_{n-1} ~=~ 1 \,, \quad x_n=1 \,.
\eeq
We also set $D_0^{(l)} = 1$.

\begin{thm}\label{Th80}
For $l \ge 3$, $n \ge 1$, there is a one-to-one correspondence
between binary vectors of length $n$ satisfying the $D_n^{(l)}$ equations
and compositions of $n$ into parts taken from the set
\beql{EqRT7}
\{1, 2, \ldots, l-2, l, l+1\} \,.
\eeq
\end{thm}

\begin{proof}
We exhibit a bijection between the two sets.
Let $c$ be a composition $n=c_1+c_2+\cdots+c_r$
into parts from \eqn{EqRT7}. For $1 \le i \le l-2$, let 
$\psi(i) = 00\ldots 01$, of length $i$ and ending with a single $1$,
let  $\psi(l) = 00\ldots 011$, of length $l$,
let $\psi(l+1) = 00\ldots 0011$, of length $l+1$,
and let 
\beql{EqRT8}
\psi(c) ~=~ \psi(c_1)\,\psi(c_2)\,\cdots\,\psi(c_r) \,,
\eeq
a binary vector of length $n$.
Note that the $\psi(i)$ for $1 \le i \le l-2$
contain runs of at most $l-3$ zeros.
Runs of $l-2$ or $l-1$ zeros in $\psi(c)$ are therefore
followed by two ones.
So conditions \eqn{EqRT5} and \eqn{EqRT6} are satisfied.
Conversely, given a binary vector satisfying the $D_n^{(l)}$ equations,
we can decompose it into substrings $\psi(i)$ by
reading it from left to right.
\end{proof}

\begin{table}[htbp]
$$
\begin{array}{|c|cccc|} \hline
n \backslash l
 & 3 & 4 & 5 & 6 \\
\hline
 0 & 1 & 1 & 1 & 1  \\
 1 & 1 & 1 & 1 & 1  \\
 2 & 1 & 2 & 2 & 2  \\
 3 & 2 & 3 & 4 & 4  \\
 4 & 4 & 6 & 7 & 8  \\
 5 & 6 & 11 & 14 & 15  \\
 6 & 9 & 20 & 27 & 30  \\
 7 & 15 & 36 & 51 & 59  \\
 8 & 25 & 65 & 98 & 115  \\
\hline
\end{array}
$$
\caption{Table of $D_{n}^{(l)}$ (the columns are A006498, A079976, A079968, A189101).}
\label{TabDhat}
\end{table}

The generating function for the $D_n^{(l)}$ follows
immediately from the theorem:

\begin{cor}\label{CorgfD}
For $l \ge 3$, the numbers $D_n^{(l)}$ have generating function
\begin{align}\label{EqRT9}
\sD^{(l)}(z) ~:=~ \sum_{n = 0}^{\infty} D_n^{(l)} z^n  & ~=~
\frac{1}{1-(z+z^2+\cdots+z^{l-2}+z^l+z^{l+1})} \nonumber \\
        {} & ~=~
\frac{1-z}{1-2z+z^{l-1}-z^l+z^{l+2}} \,.
\end{align}
\end{cor}

Table \ref{TabDhat} shows the initial values of $D_n^{(l)}$,
computed using the generating function.
The $l=4$ and $l=5$ columns are in \cite{OEIS} as 
entries A079976 and A079968, taken from a paper by D. H. Lehmer
on enumerating permutations $(\pi_1, \ldots, \pi_n)$
with restrictions on the displacements $\pi_i-i$
(\cite{Leh69}; see also \cite{Klo09}).

We now express the numbers $M_h^{(l)}$ in terms of the $D_n^{(l)}$.
We consider the $l$ values $x_1, \ldots, x_l$ in
a solution to the $M_h^{(l)}$ equations, 
and classify them according to the number of leading zeros.
There are just $l-1$ possibilities, as shown in Table \ref{TabM2D},
and in each case the $M_h^{(l)}$ equations reduce to an instance of
the $D_n^{(l)}$ equations.
For example, if $x_1=1$ the $M_h^{(l)}$ equations reduce to an instance of
the $D_{h-3}^{(l)}$ equations.
(E.g., if $l=5$ and we set $x_1=1$, equations \eqn{EqRT2}, \eqn{EqRT3}, \eqn{EqRT4}
become $x_{v-2} \CA x_{v-1} \CA x_v=1$,
$x_{2} \CA x_{3} \CA x_{4} \CA x_{6}=1$,
$\ldots$,
$x_{v-4} \CA x_{v-3} \CA x_{v-2} \CA x_{v}=1$,
which, if we subtract $1$ from each subscript,
are the equations for $D_{v}^{(5)}$, that is, $D_{h-3}^{(5)}$.)

\begin{table}[htbp]
$$
\begin{array}{|l|c|} \hline
\mbox{Setting, in } M_h^{(l)}, & \mbox{leads~to} \\
\hline
x_1=1 & D_{h-3}^{(l)} \\
x_1=0, x_2=1 & D_{h-4}^{(l)} \\
x_1=x_2=0, x_3=1 & D_{h-5}^{(l)} \\
\cdots & \cdots \\
x_1=\cdots=x_{l-4}=0, x_{l-3}=1 & D_{h-l+1}^{(l)} \\
x_1=\cdots=x_{l-3}=0, x_{l-2}=x_{l-1}=1 & D_{h-l-1}^{(l)} \\
x_1=\cdots=x_{l-2}=0, x_{l-1}=x_{l}=1 & D_{h-l-2}^{(l)} \\
\hline
\end{array}
$$
\caption{Expressing $M_h^{(l)}$ in terms of $D_n^{(l)}$.}
\label{TabM2D}
\end{table}

We have therefore shown that for $l \ge 3$, $h \ge 3$,
\beql{EqM2D}
M_{h}^{(l)} 
~=~
D_{h-3}^{(l)} +
D_{h-4}^{(l)} + 
D_{h-5}^{(l)} + \cdots +
D_{h-l+1}^{(l)} +
D_{h-l-1}^{(l)} +
D_{h-l-2}^{(l)} \,.
\eeq

Table \ref{TabM2D2} illustrates \eqn{EqM2D} in the case $l=4$.

\begin{table}[htbp]
$$
\begin{array}{|c|c|ccc|} \hline
h & M_{h}^{(4)} & D_{h-3}^{(4)} & D_{h-5}^{(4)} & D_{h-6}^{(4)} \\
\hline
 3 & 1 & 1 & - & -  \\
 4 & 1 & 1 & - & -  \\
 5 & 3 & 2 &   1 & -  \\
 6 & 5 & 3 &   1 & 1  \\
 7 & 9 & 6 &   2 & 1  \\
 8 & 16 & 11 & 3 & 2  \\
 9 & 29 & 20 & 6 & 3  \\
\hline
\end{array}
$$
\caption{Illustrating $M_{h}^{(4)} = D_{h-3}^{(4)} + D_{h-5}^{(4)} + D_{h-6}^{(4)}$.}
\label{TabM2D2}
\end{table}

From \eqn{EqM2D} and Corollary \ref{CorgfD}, and taking into
account the values of $M_h^{(l)}$ for $h<3$, we obtain:

\begin{thm}\label{ThM2}
For $l \ge 3$, 
\begin{align}\label{EqM2}
\sM^{(l)}(z) ~:=~ \sum_{h = 1}^{\infty} M_h^{(l)} z^h  & ~=~
z+(z^3+z^4+\cdots+z^{l-1}+z^{l+1}+z^{l+2})\,\sD^{(l)}(z) \nonumber \\
        {} & ~=~ \frac{z(1-z)}{1-(z+z^2+\cdots+z^{l-2}+z^l+z^{l+1})} \nonumber \\
        {} & ~=~ \frac{z(1-z)^2}{1-2z+z^{l-1}-z^l+z^{l+2}} \,.
\end{align}
\end{thm}

It is now straightforward to obtain the recurrence for $M_h^{(l)}$ from
the generating function in the second line of the display.
We omit the proof.

\begin{cor}\label{Cor7}
For $l\ge 3$, $M_h^{(l)}$ satisfies the recurrence \eqn{EqMrec} with
initial conditions \eqn{EqRT0}, \eqn{EqRT1}.
\end{cor}

We can now give the proof of Theorem \ref{Th62}.
From the definition of $M_h^{(l)}$, we have
$$
d_2(2^k-3) ~=~ \sum_{l=1}^{k}M_{k-l+1}^{(l)}
\,.
$$
That is, $d_2(2^k-3)$ is the sum of 
the coefficient of $z^k$ in $\sM^{(1)}(z)$,
the coefficient of $z^{k-1}$ in $\sM^{(2)}(z)$, $\ldots$,
and the coefficient of $z^1$ in $\sM^{(k)}(z)$.
In other words, $d_2(2^k-3)$ is the coefficient of $z^k$ in
$$
\sM^{(1)}(z)+
z \sM^{(2)}(z)+
z^2 \sM^{(3)}(z)+\cdots
+z^{k-1}\sM^{(k)}(z) \,,
$$
and now \eqn{Eq62} follows from $\sM^{(1)}(z) = z + z^3/(1-z)$,
$\sM^{(2)}(z)=0$, and Theorem \ref{ThM2}.
This completes the proof of Theorem \ref{Th62}.

\vspace*{+.2in}

\noindent
\textbf{Remark.}
The $l=3$ column of the $M_h^{(l)}$ table
(Table \ref{Tab61}) is an interesting sequence
in its own right.\footnote{It is entry A070550 in \cite{OEIS},
which contains a comment by Ed Pegg, Jr., that it
arises in the analysis of Penney's game.}
To analyze it directly, first consider
the system of simultaneous Boolean equations
\beql{EqFib}
x_1 \CA x_2=x_2 \CA x_3 \CA \cdots \CA x_{n-1} \CA x_n=1 \,,
\eeq
for $n \ge 2$, involving a chain of linked pairs of variables.
An easy induction shows that the number of solutions
is the Fibonacci number $F_{n+2}$ (cf. A000045).\footnote{This
result could also be obtained by the Goulden-Jackson
cluster method, as implemented by Noonan and Zeilberger
\cite{GoJa83}, \cite{NoZe99}.}  
Second, the equations for $M_h^{(3)}$, \eqn{EqRT3} and \eqn{Eql3},
break up into two disjoint chains like \eqn{EqFib},
and we find that
\beql{EqFib2}
M_{h}^{(3)} ~=~
\begin{cases}
F_{(n-2)/2} F_{n/2},  &\text{if $n$ is even}, \\
F_{(n-3)/2} F_{(n+1)/2},  &\text{if $n$ is odd}. \\
\end{cases}
\eeq
From \eqn{EqFib2} we can derive the recurrence
$M_h^{(3)} = M_{h-1}^{(3)} + M_{h-3}^{(3)} + M_{h-4}^{(3)}$
and the generating function 
\beql{Eq31}
\sM^{(3)}(z) ~=~ \frac{z(1-z)}{1-z-z^3-z^4} \,,
\eeq
in agreement with \eqn{EqMrec} and \eqn{EqM2}.

\vspace*{+.2in}

The final result in this section will describe the asymptotic 
behavior of the sequence $d_2(2^k-3)$.
When investigating Conjectures \ref{Conj2} and \ref{Conj3},
we observed that among all numbers $n$ with $k$ binary
digits, the number $2^k-1$ was the clear winner, with values
close to the estimate \eqn{EqKemp}.
The runners-up, a long way behind,
were $2^k-3$ and $2^k-2^{k-2}-1$ (and sometimes other values of $n$),
all with the same number of dismal divisors, for which the number
of dismal divisors appeared to be converging
to one-fifth of the number of divisors
of the winner, or in other words it appeared that
\beql{EqA2a}
\frac{d_2(2^k-3)}{d_2(2^k-1)} \, ~\rightarrow~ \, \frac{1}{5}, \mbox{~as~} k \rightarrow \infty \,.
\eeq

We will now establish this from the generating function \eqn{Eq62}.

\begin{thm}\label{ThM3}
\beql{EqA2b}
d_2(2^k-3) ~\sim~ \frac{2^{k}}{5k \log 2}~(1 + \bar{\Theta}_k), 
\mbox{~as~} k \rightarrow \infty \,,
\eeq
where $\bar{\Theta}_k$ is a bounded oscillating function
with $|\bar{\Theta}_k| < 10^{-5}$.
\end{thm}

\begin{proof}
Our proof is modeled on Knopfmacher and Robbins's proof \cite{KnRo05}
of \eqn{EqKemp}, which uses the method of Mellin transforms
as presented by Flajolet, Gourdon, and Dumas \cite{FGD95}.
We will indicate how the Knopfmacher-Robbins proof
can be reworded so as to establish \eqn{EqKemp} and \eqn{EqA2b}
simultaneously.

Knopfmacher and Robbins work, not with \eqn{Eq3}, but with
\beql{EqA4.1}
f(z) ~:=~ \sum_{l = 1}^{\infty} \frac{(1-z) z^{l}}{1-2z+z^{l}} \,,
\eeq
which is the generating function for the number of 
compositions of $n$ into parts of which
the first is {\em strictly} greater than all the other parts (A007059).
Equations \eqn{Eq3} and \eqn{EqA4.1} basically differ just by a factor of $z$.
Then \cite{KnRo05} shows that the coefficient of $z^n$ in $f(z)$ is
\beql{EqA4.2}
\frac{2^{n-1}}{n \log 2} (1 + \Theta) \,,
\eeq
for some small oscillating function $\Theta$,
which implies \eqn{EqKemp}.

So as to have a function with the same form as \eqn{EqA4.1},
we consider, not \eqn{Eq62}, but
\beql{EqA5.1}
f(z) ~:=~ 
\sum_{l=2}^{\infty} \frac{(1-z)^2 z^l}{1-2z+z^{l}-z^{l+1}+z^{l+3}}  \,,
\eeq
and we will show that the coefficient of $z^n$ is
\beql{EqA5.2}
\frac{2^{n+1}}{5n \log 2} (1 + \Theta) \,,
\eeq
for some (different) small oscillating function $\Theta$,
which implies \eqn{EqA2b}.
We can change the lower index of summation in \eqn{EqA5.1} from
$2$ to $1$,  since the $l=1$ term is the 
generating function for the Padovan sequence (A000931), which grows at a much 
slower rate than \eqn{EqA5.2}.

In what follows, we simply record how the expressions in Knopfmacher
and Robbins's proof \cite{KnRo05} need to be modified so as to apply
simultaneously to \eqn{EqA4.1}, which we refer to as case I,
and \eqn{EqA5.1}, which we call case II.
We follow Knopfmacher and Robbins's notation, except that we use $j$ and $m$ as local
variables, rather than $k$, to avoid confusion with the $k$
in the statement of the theorem.
Some typographical errors in \cite{KnRo05} have been silently corrected.

Let $\rho_j$ denote the smallest root of the denominator
of the $j$-th summand in $f(z)$  that lies between $0$ and $1$.
Then 
\beql{EqA13}
\rho_j ~=~ \frac{1}{2} \, \big( 1 + \tau 2^{-j} + O(j 2^{-2j}) \big) \,,
\eeq
where $\tau=1$ (case I) or $5/8$ (case II).

Let $q_{n,j}$ denote the coefficient of $z^n$
in the $j$-th summand in $f(z)$.
Then
\begin{align}
q_{n,j} & ~\approx~ 2^{n-j-\epsilon} \, \big( 1 - \frac{\tau}{2^j} \big)^n \nonumber \\
     {} & ~\approx~ 2^{n-j-\epsilon} \, e^{-\tau n / 2^j} \,,
\end{align}
where $\epsilon = 1$ in case I or $2$ in case II. 
Next, $f_n$, the coefficient of $z^n$ in $f(z)$, is
$$
2^{n-\epsilon} ~ \Bigg( \sum_{j = 2}^{\infty} 2^{-j} e^{-\tau n / 2^j} + o(1) \Bigg) \,.
$$
Let 
$$
g(x) ~:=~ \sum_{j = 2}^{\infty} 2^{-j} e^{-\tau x / 2^j} \,.
$$
The Mellin transform of $g(x)$ is
$$
g^{\ast}(s)  ~:=~ \frac{1}{\tau^s} \, \frac{2^{2(s-1)}}{1-2^{s-1}} \, \Gamma(s),
\quad 0 < \Re (s) < 1 \,.
$$
To compute $f_n$ we use (following Knopfmacher and Robbins) 
the Mellin inversion formula
$$
g(x) ~=~ \frac{1}{2 \pi i} \int_{1/2 - i \infty}^{1/2 + i \infty} x^{-s} g^{\ast}(s) \, ds \,. 
$$
Now $g^{\ast}(s) x^{-s}$ has a simple pole at $s = 1 + \chi_m$, for each $m \in \ZZ$,
where $\chi_m = 2 \pi i m /\log 2$,
with residue
$$
- \, \frac{1}{x \log 2} \,
\frac{\Gamma(1+\chi_m) \, e^{-2 \pi i m \log _2 x}}{\tau^{1+\chi_m}} \,.
$$
After combining the contributions from all the poles, we have
\beql{Eq96.3}
f_n ~=~  
2^{n-\epsilon} \, \sum_{m = -\infty}^{\infty} \, \frac{1}{n \log 2} \, 
\frac{\Gamma(1+ 2 \pi i m / \log 2) \, 
e^{-2 \pi i m \log _2 n}}{\tau^{1+ 2 \pi i m / \log 2}} \,.
\eeq
The term for $m=0$ dominates, and we obtain the desired results
\eqn{EqA4.2} and \eqn{EqA5.2}.
\end{proof}

\section{The sum of dismal divisors}\label{SecSum}

We briefly discuss the dismal sum-of-divisors 
function $\sigma _b(n)$
(see A188548, A190632, and A087416 for bases $2$, $3$, and $10$).

\begin{thm}\label{ThBase1}
In any base $b \ge 2$, if $\len_b(n)=k$, then 
\beql{EqBase0}
n \le \sigma_b(n) \le b^k-1\,,
\eeq
and $\sigma_b(n) = n$ if and only if $n \equiv b-1 \pmod{b}$.
\end{thm}

\begin{proof}
The first assertion follows because $n$ divides itself,
and no divisor has length greater than $k$.
If $n \not\equiv b-1 \pmod{b}$ then since $b-1$ is the multiplicative
unit, $\sigma_b(n) \neq n$.
Suppose $n \equiv b-1 \pmod{b}$ and $p$ is a divisor of $n$,
with say $p \CMsub _b\, q = n$.
Both $p$ and $q$ must end with $\beta$.
From the long multiplication tableau, $p \ll _b n$,
so $p \CAsub _b\, n = n$, and therefore $\sigma_b(n) = n$.
\end{proof}

In ordinary arithmetic, a number $n$ is perfect
if its sum of divisors is $2n$. In dismal arithmetic,
$n \CAsub _b \, n = n$.
So the second part of the theorem might,
by a stretch, be interpreted as saying
that the numbers congruent to $b-1 \pmod{b}$
are the base $b$ perfect dismal numbers.

In base $2$, then, $\sigma_2(n) = n$ if and only if $n$ is odd.
Examination of the data shows that,
if $n$ is even, with $\len_2(n)=k$,
often $\sigma_2(n)$ takes its maximal value, $2^k-1$.
Table \ref{TabBase2} shows the first few exceptions,
which are characterized in the next theorem.

\begin{table}[h]
$$
\begin{array}{|rr|rr|} \hline
n & \sigma_{2}(n) & n & \sigma_{2}(n) \\
\hline
10010 & 11011 & 1001010 & 1101111 \\
100010 & 110011 & 1001110 & 1101111 \\
100110 & 110111 & 1010010 & 1111011 \\
110010 & 111011 & 1100010 & 1110011 \\
1000010 & 1100011 & 1100110 & 1110111 \\
1000100 & 1110111 & 1110010 & 1111011 \\
1000110 & 1100111 & \ldots\ldots & \ldots\ldots \\
\hline
\end{array}
$$ 
\caption{Even numbers $n$ such that $\sigma_2(n)$
is not of the form $11\ldots 1|_2$
(A190149--A190151). Both $n$ and $\sigma_2(n)$ are written in base $2$.}
\label{TabBase2}
\end{table}

\begin{thm}\label{ThBase2}
Suppose $n=2^rm$ with $r \ge 1$, $m$ odd, and $\len_2(n)=k$.
Then 
$$
\sigma_2(n) ~=~ 2^k-1  ~=~  \underbrace{11\ldots1}_{k}|_2 
$$
unless the binary expansion of $m$ contains a run of more than 
$r$ consecutive zeros.
\end{thm}

\begin{proof}
Since $m$ is odd, $\sigma_2(m)=m$.
Therefore
$$
\sigma_2(n) ~=~ 
m|_2 ~ \CAsub _2 \, ~
m0|_2 ~ \CAsub _2 \, ~
m00|_2 ~ \CAsub _2 \, ~ \cdots ~ \CAsub _2 \, ~
m\underbrace{00\ldots 0}_r|_2 \,,
$$
and any string $\underbrace{00\ldots0}_i1$
in $m$ will become $\underbrace{11\ldots1}_i1$
in $\sigma_2(n)$ unless $i$ exceeds $r$.
\end{proof}

The first entry in Table \ref{TabBase2} 
is explained by the fact that $n=10010|_2$, $r=1$,
$m=1001|_2$, and $m$ contains a run of two zeros.

We also considered two other possible definitions 
of perfect numbers: (i) $n$ is perfect in base $b \ge 3$ if
the dismal sum of the dismal divisors of $n$ is equal to $2 \CMsub _b\,n$.
We leave it to the reader to verify that for this to happen,
$b$ must be $3$, and then $n$ is perfect if and 
only if $n \equiv 2 \bmod{4}$. 
(ii) $n$ is perfect in base $b \ge 2$ if
the dismal sum of the dismal divisors of $n$ different from $n$
is equal to $n$. But here $n$ cannot be $b-1$, so $b-1$ is a divisor,
and $n$ ends with $b-1$. This implies that $n$ has no divisors
of length $\len_b(n)$ except $n$ itself,
so the sum cannot equal $n$, and therefore
no such $n$ exists.

We end this section with a conjecture (see A186442):
\begin{conj}
\label{Conj4}
For all $n>1$, $d_{10}(n) < \sigma_{10}(n)$.
\end{conj}

\section{Dismal partitions}\label{SecPart}
Since $n \CAsub _b\, n = n$, it only makes sense to
consider partitions into distinct parts (otherwise every number has
infinitely many different partitions).
We define $p_b(n)$ to be the number of ways of writing
\beql{EqPart1}
n ~=~ m_1 \CAsub _b\, m_2 \CAsub _b\, \cdots \CAsub _b\, m_l \,,
\eeq
for some $l \ge 1$ and distinct positive integers $m_i$,
without regard to the order of summation. We set $p_b(0)=1$ by
convention.

For example, $p_3(7) = p_3(21|_3) = 22$, since (working in base $3$)
$21$ is equal to
$21 ~\CAsub_3$ any subset of $\{20, 11, 10, 1\}$ ($16$ solutions),
$20 ~\CAsub_3\, 11 ~\CAsub_3$ any subset of $\{10, 1\}$ ($4$ solutions),
and $20 ~\CAsub_3\, 1 ~\CAsub_3$ any subset of $\{10\}$ ($2$ solutions),
for a total of $22$ solutions. 

\noindent
\textbf{Remarks.}

(i) Permuting the digits of $n$ does not change $p_b(n)$.

(ii) Any zero digits in $n$ can be ignored. If $n'$ is the
base $b$ number obtained by dropping any $n_i$'s that are zero,
$p_b(n')=p_b(n)$.

(iii) Although we will not make any use of it, there is a generating function
for the $p_b(n)$ analogous to that for the classical case.
If we interpret $z^m \CMsub _b\, z^n$ to mean $z^{m  \CAsub _b\, n}$,
then we have the formal power series
$$
1 + p_b(1)z + p_b(2)z^2 + p_b(3)z^3 + \cdots ~=~ 
(1+z) \CMsub _b\,
(1+z^2) \CMsub _b\,
(1+z^3) \CMsub _b\, \cdots \,.
$$
 
(iv) The sequences $p_2(n)$ and $p_{10}(n)$
form entries A054244 and A087079 in \cite{OEIS},
contributed by the second author in 2000 and 2003, respectively.

In the remainder of this section we index the 
digits of $n$ by $\{1, 2, \ldots, n\}$,
in order to simplify the discussion of subsets
of these indices.

\begin{thm}\label{ThPart1}
If $n=n_1 n_2 \ldots n_k|_2$, $n_i \in \{0,1\}$,
and the binary weight of $n$ is $w$, then
$p_2(n)$ is equal to the number of set-covers of a labeled $w$-set
by nonempty sets (cf. A003465, \cite[p.~165]{Comtet}), that is,
\beql{EqPart2}
p_2(n) ~=~ \frac{1}{2}\, \sum_{i=0}^{w} (-1)^{w-i} \binom{w}{i} 2^{2^i} \,.
\eeq
\end{thm}

\begin{proof}
From Remark (ii), $p_2(n) = p_2(\underbrace{11\ldots1}_w|_2)$.
There is an obvious one-to-one correspondence between
collections of distinct nonempty subsets of $\{1,\ldots,w\}$ whose
union is $\{1,\ldots,w\}$ and sets of distinct nonzero binary vectors
whose dismal sum is $\underbrace{11\ldots1}_w|_2$.
\end{proof}

\begin{thm}\label{ThPart2}
If $n=n_1 n_2 \ldots n_k|_b$, $0 \le n_i \le b-1$, then
\beql{EqPart3}
p_b(n) ~=~ \frac{1}{2}\, \sum_{S \subseteq \{1,\ldots,k\}} 
(-1)^{|S|} \, 2^{ \prod_{i} (n_i + \epsilon_i)} \,,
\eeq
where $\epsilon_i = 0$ if $i \in S$, $\epsilon_i = 1$ if $i \notin S$.
\end{thm}

\begin{proof}
The set $\Omega_n$ of $x \ll_b n$ is a partially ordered set
(with respect to the operator $\ll_b$) with M\"{o}bius function 
given by \cite[\S3.8.4]{Stanley}
\beql{EqPart4}
\mu(x_1 \ldots x_k|_b, \, y_1 \ldots y_k|_b) ~=~
\begin{cases}
(-1)^{ \sum_i (y_i-x_i)},  & \mbox{if~} y_i-x_i = 0 \mbox{~or~} 1 \mbox{\,for~all\,} i, \\
0,  & \text{otherwise}. \\
\end{cases}
\eeq
Every subset of $\Omega_n \setminus \{0\}$ 
has dismal sum equal to some number $\ll_b n$, so we have
$$
\sum_{x \ll_b n} p_b(x) ~=~ 2^{ \prod_i (n_i+1) -1} \,.
$$
From the M\"{o}bius inversion formula \cite[\S3.7.1]{Stanley}
we get
$$
p_b(n) ~=~ \frac{1}{2}\, \sum_{S \subseteq \{1,\ldots,k\}}
(-1)^{|S|} \, 2^{ \prod_{i \in S} n_i  \prod_{j \notin S} (n_j+1)} \,,
$$
which implies \eqn{EqPart3}.
\end{proof}

Theorem \ref{ThPart2} reduces to Theorem \ref{ThPart1} if all $n_i$ are $0$ or $1$.

\begin{cor}\label{CorPart3}
Suppose $n = n_1 n_2 \ldots n_k|_b$,
and let $y$ be the ordinary product $n_1 \times n_2 \times \cdots \times n_k$.
Then $p_b(n)$ is divisible (in ordinary arithmetic!) by $2^{y-1}$.
\end{cor}

\begin{cor}\label{CorPart4}
For a single-digit number $n = n_1|_b$, $p_b(n)=2^{n_1-1}$.
For a two-digit number $n = n_1 n_2|_b$, 
\beql{EqPart6}
p_b(n) ~=~ 2^{(n_1+1)(n_2+1) -2} ~-~ 2^{n_1 n_2 - 1}(2^{n_1}-1)(2^{n_2}-1) \,.
\eeq
\end{cor}

Eq. \eqn{EqPart6} was found by Wasserman \cite[entry A087079]{OEIS}.

It follows from the above discussion that in any base $b$,
the only numbers $n$ such that $p_b(n)=1$ are
$0|_b$, $1|_b$, $10|_b$, $100|_b$, $1000|_b , \ldots$,
and that all other numbers  $n$ have the property that
the dismal sum of the numbers $x \ll_b n$ is $n$.
These two classes might be called ``additive primes''
and ``additive perfect numbers.''

\section{Conclusion and future explorations}\label{SecConc}
We have attempted to show that dismal arithmetic, despite its simple
definition, is worth studying for the interesting problems that arise.
We have left many questions unanswered:
the ``prime number theorem'' of Conjecture \ref{ConjPrime},
the questions about the numbers of divisors stated in
Conjectures \ref{Conj1}-\ref{Conj3} (in particular,
is it true that $11\ldots1|_{10}$ has more
divisors than any other base $10$ number
with the same number of digits?),
the base $10$ dismal analog of $d(n) \le \sigma(n)$ (Conjecture \ref{Conj4}),
and the two questions about dismal squares at the end 
of \S\ref{SecSquares}---in particular, is there a 
recurrence for the sequence \eqn{Eq191701}?
There are numerous other questions that we have not investigated (for
example, if $x \ll_b y$, what can be said about $x \CMsub _b\, y$?).

We have made no mention of the complexity of deciding
if a number is a dismal prime, or of finding dismal factorizations.
In base $2$ such questions reduce to solving a set of 
simultaneous quadratic Boolean equations,
where a typical equation might be
$$
(x_0 \CMsub _2\, y_4) ~ \CAsub _2 ~
(x_1 \CMsub _2\, y_3) ~ \CAsub _2 ~ \cdots ~ \CAsub _2 ~
(x_4 \CMsub _2\, y_0)  ~=~ 0~ (\mbox{or }1) \,.
$$
This becomes a question about the satisfiability of
a complicated Boolean expression, and is likely
to be hard to solve in general \cite{GaJo79}.

While we have focused on the cases $b=2$ and $b=10$, 
it would be nice to better
understand the qualitative differences across a wider range of bases.  For
example, while $b=2$ is a kind of Boolean arithmetic, does $b=3$
correspond to a three-valued logic? 
Do odd $b$ behave differently from even $b$?
More generally, what other interesting mathematical structures might be
modeled by dismal arithmetic?

\section*{Acknowledgments}
We thank Adam Jobson for computing an extended version of Table 
\ref{Tab61}, which was helpful in guessing the recurrence \eqn{EqMrec}.
Most of the our computations were programmed in C++, Fortran, Lisp, and Maple.
We also made use of Mathematica's {\texttt SatisfiabilityCount}
command, and we thank Michael Somos for telling us about it.
We are also grateful to Doron Zeilberger for telling us
about his work with John Noonan \cite{NoZe99}
on implementing the Goulden-Jackson cluster method.
The OEIS \cite{OEIS} repaid us several times during the course
of this work by providing valuable hints, notably
from the entries A007059, A067399,
A070550, A079500, A156041, and A164387.

\bigskip
\hrule
\bigskip

\noindent 2010 {\it Mathematics Subject Classification}:
Primary 06A06, 11A25, 11A63; Secondary 11N37.

\noindent \emph{Keywords: } 
Carryless arithmetic, squares, primes, divisors,
compositions, partitions, asymptotic expansions, Mellin transform.


\bigskip
\hrule
\bigskip

\vspace*{+.1in}
\noindent

\vskip .1in

\end{document}